\documentclass{amsart}

\theoremstyle{plain}
\newtheorem{theorem}{Theorem}[section]
\newtheorem{proposition}[theorem]{Proposition}

\newtheorem{corollary}[theorem]{Corollary}

\theoremstyle{definition}
\newtheorem{assumption}{Assumption}

\theoremstyle{remark}
\newtheorem{remark}{Remark}

\title[EK ASYMPTOTICS FOR INFINITE-DIMENSIONAL GRADIENT SYSTEMS]{Eyring--Kramers asymptotics for infinite-dimensional stochastic gradient systems}

\author[M. Brooks]{Morris Brooks}
\address{Institute of Mathematics, University of Zurich, Winterthurerstrasse 190, CH-8057 Zurich, Switzerland}
\email{morris.brooks@math.uzh.ch}

\author[G. Di Gesu]{Giacomo Di Gesu}
\address{Dipartimento di Matematica, Sapienza Universit\`a di Roma, Piazzale Aldo Moro 5, 00185 Roma, Italy}
\email{giacomo.digesu@uniroma1.it}


\begin{document}

\begin{abstract}

We study small-noise asymptotics     
 for a class of reversible stochastic evolution equations in infinite dimensions. The dynamics are of the form
\[
dX_t=-A\nabla F(X_t)\,dt+\sqrt{2\beta^{-1}A}\,dW_t,
\]
where $F$ is a regular multi-well potential, $A$ is a selfadjoint mobility operator,  $W$ is a cylindrical Brownian motion  and $\beta\gg 1$ is the inverse noise strength. 
The invariant measure is a Gibbs perturbation of a Gaussian reference measure, and the resulting framework covers, in particular, the stochastic Allen--Cahn and stochastic Cahn--Hilliard equations on bounded intervals. 
In the double-well case, we derive a sharp asymptotic formula for the first nonzero eigenvalue of the generator. 
This gives an infinite-dimensional Eyring--Kramers law for the spectral gap, with exponential rate determined by the communication height and leading prefactor determined by the local quadratic behavior at the relevant minima and saddle points.
Our approach provides a general strategy for lifting finite-dimensional Eyring--Kramers analysis to infinite-dimensional stochastic gradient systems.

\end{abstract}

\maketitle

\section{Introduction}

\noindent
We consider stochastic gradient dynamics formally described by
 \begin{equation}  \label{eq} 
  dX_t=-A\nabla F(X_t) \, dt+\sqrt{2\beta^{-1}A} \, dW_t .\end{equation} Here $W$ is a cylindrical Brownian motion on a separable Hilbert space $H$, the potential $F \in C^2(H)$ 
satisfies suitable coercivity assumptions at infinity 
 and $A$ is a selfadjoint uniformly positive operator on $H$. The parameter $\beta>0$ is the inverse noise strength, and we are interested in the small-noise regime $\beta\to\infty$.

Heuristically, the drift in \eqref{eq} drives the dynamics towards the local minima of the functional $F$, while the noise creates rare transitions between different wells. When $F$ has several local minima, this competition gives rise to metastable behavior: on moderate time scales the process appears trapped near one minimum, whereas on exponentially long time scales transitions between different wells become visible. 
A central quantity encoding this phenomenon is the spectral gap $\lambda_\beta$, namely the first nonzero eigenvalue of the generator. Its inverse $\lambda_\beta^{-1}$ corresponds to the relaxation time of the dynamics and is therefore expected to be exponentially large in the small-noise regime. More generally, if the energy landscape has $k_0$ local minima, one expects a cluster of $k_0$ low-lying eigenvalues: the zero eigenvalue corresponding to equilibrium, together with $k_0-1$ positive eigenvalues which are exponentially small in $\beta$. These eigenvalues describe the effective transitions between the metastable wells, while the rest of the spectrum remains bounded away from zero and corresponds to faster relaxation modes within the wells.

When $\dim(H)<\infty$, this picture is by now classical. Large deviations theory identifies the exponential scale of rare transitions \cite{FreidlinWentzell1998}, while sharp asymptotic estimates can be obtained through potential-theoretic and semiclassical methods; see, for instance, \cite{BovierEckhoffGayrardKlein2004,HelfferKleinNier2004,Eckhoff2005,MenzSchlichting2014,DGLLLPN2019,AvelinJulinViitasaari2023}, and see also \cite{LandimMarianiSeo2019,LeeSeo2022,LePeutrecMichel2020} for non-reversible extensions.
For example, suppose that $H=\mathbb R^d$, that $A$ is a symmetric positive definite matrix, and that $F$ has exactly two local minima $x_-$ and $x_+$ with $F(x_-)<F(x_+)$, separated by a unique nondegenerate saddle point $z$ of index one. Let $-\alpha<0$ be the unique negative eigenvalue of 
 the Hessian of $F$ with respect to the modified scalar product $\langle A^{-1}\cdot, \cdot \rangle$ and
 assume that $\nabla^2F(x_\pm)>0$. 
Then, as $\beta\to\infty$,  the spectral gap satisfies 
the Eyring--Kramers law
\begin{equation}    \label{FDEK} \lambda_\beta = \frac{\alpha}{2\pi} \sqrt{ \tfrac{\det \nabla^2F(x_+)} {|\det \nabla^2F(z)|} } e^{-\beta[ F(z)-F(x_+)]}  \,   \left(1+o(1) \right). 
\end{equation}
Further, $ \lambda_\beta$ is  separated from the rest of the spectrum by a gap of order one.
The determinant ratio in~\eqref{FDEK} comes from Gaussian approximations of the invariant measure $Z_\beta^{-1}e^{-\beta F(x)} dx$ near the minimum $x_+$ and the saddle $z$, while the factor $\alpha/(2\pi)$ accounts for the unstable direction at the saddle.

\

The goal of this paper is to establish an analogue of this picture in infinite dimensions. In this setting, the cylindrical Brownian motion in \eqref{eq} is not $H$-valued, but can be realized only after embedding $H$ into a larger separable Banach space $B$.
Our basic structural assumption is that $F$ is a sufficiently regular perturbation of the quadratic Gaussian energy. More precisely, we assume that there exists a function $U\in C^2(B)$ such that
\[
U|_{H}
=
F-\frac12|\cdot|_H^2 .
\]
As a consequence of this assumption the invariant measure, formally proportional to $e^{-\beta F(x)} dx$, can be represented as the Gibbs perturbation
\[
d\mu_\beta
=
Z_\beta^{-1}e^{-\beta U} d\gamma_\beta ,
\]
where $\gamma_\beta$ is obtained by rescaling the centered Gaussian measure on $B$ with Cameron-Martin space $H$. 
Despite this enlargement of the state space, and despite the fact that $\mu_\beta (H) = \gamma_\beta(H)=0$ in infinite dimensions, the metastable behavior is still governed by the geometry of the energy landscape on the Cameron--Martin space $H$; in this sense, the relevant low-energy picture is the one that would arise if everything happened on $H$. To make this principle precise, we prove two main results. 

\

{\bf Main results.} The first one, Theorem~\ref{Thm:Main_Abstract_Theorem}, is a spectral separation theorem. 
Suppose that $F$ has $k_0$ local minima and satisfies suitable coercivity and nondegeneracy assumptions, see Assumption~\ref{Ass_for_theorem_I} below. 
 Then the $k_0$-th eigenvalue is bounded away from zero uniformly in the small-noise limit: 
\[ \varliminf_{\beta\to\infty}\lambda_\beta^{(k_0)}>0 . \]
 Equivalently, at most $k_0$ eigenvalues vanish as $\beta\to\infty$. Since each local minimum gives rise to a quasimode with exponentially small energy, this identifies the number of exponentially small eigenvalues with the number of local minima. 
 The proof  is based on the variational characterization of the eigenvalues through the infinite-dimensional Dirichlet form~\cite{AlbeverioRockner1990, AlbeverioRockner1991, MaRockner1992}
 \[          \mathcal E_\beta[f]    =  \frac 1\beta   \int_B    \|  \sqrt A \nabla f\|_H^2     \,  d\mu_\beta                   ,        \]
 and on a localization procedure: away from the critical set, after a ground-state transformation, a large-deviation estimate yields a uniform lower bound, with the logarithmic Sobolev inequality for the Gaussian reference measure $\gamma_\beta$ as a key ingredient; near minima, the analysis reduces to convex models, which yield one low-energy direction, while near saddles and higher-index critical points one instead exploits an unstable direction to rule out additional low-energy modes.

 Our second main result, Theorem~\ref{Th:precise_Asy}, gives the sharp Eyring--Kramers asymptotics in the double-well case, i.e. 
 assuming that $F$ has exactly two local minima $x_-$ and $x_+$. For example, if $F(x_-)<F(x_+)$ and there is a unique nondegenerate saddle point $z$ of index one, 
 Theorem~\ref{Th:precise_Asy}  states that~\eqref{FDEK} still holds, with 
 $-\alpha$ the unique negative eigenvalue of the selfadjoint operator
associated with  $\nabla^2 F(z)$  in the metric induced by the modified scalar product $\langle A^{-1} \cdot, \cdot\rangle_H$. We emphasize that the determinants 
appearing in the prefactor of $\eqref{FDEK}$ are well-defined as Fredholm determinants of trace class perturbations of the identity. Indeed, under our assumptions,
\[     \nabla^2 F(x)     =   \text{Id}     +   \nabla^2 U(x) ,         \    \   \     \text{ for every } x\in H,  \]
and $\nabla^2 U(x)$ is automatically trace class by a general result of Goodman.
Theorem~\ref{Th:precise_Asy} also describes the modifications  that occur to the prefactor in the case $F(x_-)= F(x_+)$ and in the case of several relevant saddle points. 
The proof is based on the choice of a suitable approximate eigenfunction, a precise estimate of its Dirichlet energy and an application of the spectral separation given by Theorem~\ref{Thm:Main_Abstract_Theorem} together with a Kato--Temple type inequality.

\

{\bf Applications.} Our framework is designed to provide a unified treatment of a broad class of reversible infinite-dimensional stochastic gradient systems. The original motivation, however, comes from two concrete stochastic partial differential equations: the stochastic Allen--Cahn and stochastic Cahn--Hilliard equations on bounded intervals $(0,\ell)$.
They are given, respectively, by 
\begin{equation}    \label{SPDE:AC} dX_t = \bigl(\Delta X_t - X_t^3 + X_t\bigr)\,dt + \sqrt{2\beta^{-1}}\,dW^{L^2}_t 
\end{equation}
 and 
\begin{equation}    \label{SPDE:CH} 
dX_t = -\Delta\bigl(\Delta X_t - X_t^3 + X_t\bigr)\,dt + \sqrt{2\beta^{-1} (-\Delta) }\,dW^{L^2}_t   . 
\end{equation}
Here $W^{L^2}_t$ denotes a cylindrical Brownian motion on $L^2((0,\ell))$, whose formal time derivative is space-time white noise; in the Cahn--Hilliard case the zero-average constraint is imposed. 
These are prototypical phase-field models for relaxation in a double-well energy landscape. In the terminology of dynamic critical phenomena, they correspond respectively to Model A and Model B dynamics:
Allen--Cahn describes a non-conserved order parameter, whereas Cahn--Hilliard describes a conserved order-parameter dynamics and is classically associated with phase separation.
We refer, for instance, to \cite{HohenbergHalperin1977,DaPratoDebussche1996} for background on these models in the deterministic and stochastic settings. 
Related small-noise large-deviation problems for stochastic reaction-diffusion equations go back, in particular, to \cite{FarisJonaLasinio1982}; see also \cite{CerraiRockner2004}.
From the classical gradient-flow perspective, these equations correspond to evolving the same energy \[ F(x) = \int_0^\ell \left( \frac12 |x'(s)|^2 + \frac14\bigl(x(s)^2-1\bigr)^2 \right)\,ds \] with respect to different geometries: the Allen--Cahn equation is the $L^2$-gradient dynamics, whereas the Cahn--Hilliard equation is the $H^{-1}$-gradient dynamics on the zero-average subspace. A key point of our approach is that this geometric distinction is expressed relative to the Cameron--Martin space of the Gaussian reference measure, rather than relative to the usual $L^2$-based description of the equations. Indeed, for Allen--Cahn we take $H=H^1((0,\ell))$, and the classical $L^2$-geometry is recovered through the modified metric generated by $A^{-1}=(-\Delta+1)^{-1}$ on $H$.
With this choice, the term $\sqrt A\,dW_t$ in~\eqref{eq}, where $W_t$ is cylindrical on $H$, has the same law as $dW^{L^2}_t$.
Analogously, for Cahn--Hilliard we take  \[ H=\dot H^1((0,\ell)) := \left\{ x\in H^1((0,\ell)):\int_0^\ell x(s)\,ds=0 \right\}. \] 
With this choice, the mobility is $A=(-\Delta)^2$, so that the metric induced by $A^{-1}=(-\Delta)^{-2}$ is two Sobolev orders weaker than the Cameron--Martin metric. Thus, the usual $H^{-1}$-gradient structure over $L^2$ is represented in our framework as an $H^{-2}$-type geometry relative to $H=\dot H^1((0,\ell))$. 
This shift of viewpoint reflects the guiding principle that the relevant low-energy geometry is encoded on $H$, and it places both equations in the common abstract form~\eqref{eq}, where $W_t$ is cylindrical on $H$. Similar changes of Hilbertian structure appear in infinite-dimensional convergence-to-equilibrium problems; see, for instance, \cite{GrothausStilgenbauer2014}.

We apply our abstract results to both equations~\eqref{SPDE:AC},~\eqref{SPDE:CH}, see respectively Corollary~\ref{Cor:AC_result}   and Corollary~\ref{Cor:CH_result}, each providing a spectral separation result 
and Eyring-Kramers formula for the spectral gap.  Corollary~\ref{Cor:AC_result} (Allen--Cahn) generalizes a result obtained in our earlier work~\cite{BrooksDiGesu2021}, which assumed 
sufficiently small $\ell$, leading to a particularly simple saddle structure.

We remark that for the stochastic Allen--Cahn equation, Eyring--Kramers laws for metastable transition times have been studied before, for instance through finite-dimensional approximation and potential-theoretic methods; see e.g.~\cite{Barret2015,BerglundGentz2013,BerglundDiGesuWeber2017}. The more challenging Cahn--Hilliard case, by contrast, appears to be substantially less developed. Related metastability results for mass-conserving spatially extended diffusions were obtained in \cite{BerglundDutercq2016}, but that work concerns a discretized mass-conserving Allen--Cahn-type dynamics rather than the continuum stochastic Cahn--Hilliard equation considered here. 
To the best of our knowledge, Corollary~\ref{Cor:CH_result} gives the first rigorous Eyring--Kramers formula for the  stochastic Cahn--Hilliard equation~\eqref{SPDE:CH}.

We remark that the length $\ell$ of the interval is kept fixed. The joint regime in which both $\beta\to \infty$ and $\ell \to \infty$ leads to a sharp-interface problem, which is beyond the scope of the present work; 
see \cite{BertiniButtaDiGesu2025} for related asymptotics of the $\phi^4_1$ measure in this regime. 
Although we focus on Allen--Cahn and Cahn--Hilliard, the abstract framework may also apply to other one-dimensional stochastic gradient systems, such as the sine-Gordon model~\cite{Laarne2025}, or to models with non-local interaction terms.

\

{\bf Outline of the paper.} In Section~\ref{sec2} we introduce the precise setting, define the Gibbs measures and Dirichlet forms, and state the main results.
In Section~\ref{sec3} we prove Theorem~\ref{Thm:Main_Abstract_Theorem} on the spectral separation. 
Section~\ref{sec4} is devoted to the proof of Theorem~\ref{Th:precise_Asy} giving the Eyring--Kramers formula in the double-well case. 
Finally, in Section~\ref{sec5} we verify the assumptions in the main examples, namely the stochastic Allen--Cahn and stochastic Cahn--Hilliard equations, proving Corollary~\ref{Cor:AC_result} and Corollary~\ref{Cor:CH_result}.

\section{Setting and main results}   \label{sec2}

\noindent
We fix a separable Hilbert space $H$ and a functional $F:H \to\mathbb R$. It is standard ~\cite{Gross1967,Stroock2023} that $H$ admits abstract Wiener realizations: one may embed $H$ continuously and densely into a separable Banach space $B$ carrying a centered nondegenerate Gaussian measure $\gamma$ whose Cameron--Martin space is $H$.

Our standing assumption on $F$ is that it admits an abstract Wiener realization for which it can be written as a regular perturbation of the quadratic Cameron--Martin energy. More precisely, we assume that there exist an abstract Wiener space $(B,H,\gamma)$ and a function $U\in C^2(B)$ such that
\begin{equation}
\label{Eq:Def_F}
F = U|_H + \frac12 \|\cdot\|_H^2 .
\end{equation}
We fix one such realization $(B,H,\gamma)$ once and for all; the corresponding function $U$ is then uniquely determined by \eqref{Eq:Def_F}. Here and throughout, differentiability on Banach spaces is understood in the Fréchet sense.

We denote by $\|\cdot\|_B$ the norm on $B$, while $\|\cdot\|_H$ and $\langle\cdot,\cdot\rangle_H$ denote respectively the norm and inner product on $H$. We write
\[
\langle \ell,x\rangle := \ell(x),
\qquad \ell\in B^*,\ x\in B,
\]
for the dual pairing between $B^*$ and $B$.
For $f\in C^1(B)$ and $x\in B$, we write $Df(x)\in B^*$ for its Fréchet differential at $x$. The Cameron--Martin gradient $\nabla f(x)$ is the unique element of $H$ satisfying
\[
\langle \nabla f(x),h\rangle_H = Df(x)[h],
\qquad  h\in H . 
\]
The map  $x\mapsto \nabla f(x)$ is continuous whenever $f\in C^1(B)$.
If $f\in C^2(B)$ and $x\in B$, let $D^2f(x)$ be the second Fréchet differential of $f$ at $x$, viewed as a bounded symmetric bilinear form on $B\times B$. The Cameron--Martin Hessian $\nabla^2 f(x)$ is then the bounded selfadjoint operator on $H$ associated with the restricted bilinear form
\[
(h,k)\longmapsto D^2f(x)[h,k],
\qquad h,k\in H,
\]
that is,
\[
\langle \nabla^2 f(x)h,k\rangle_H
=
D^2f(x)[h,k],
\qquad h,k\in H .
\]
By Goodman's theorem, see Theorem~4.6 in~\cite{Kuo}, this operator is trace class on $H$ and 
the map     $x\mapsto \operatorname{tr}\nabla^2 f(x)$ 
is continuous. Hence the Cameron--Martin Laplacian
\[
\Delta f(x):=\operatorname{tr}\nabla^2 f(x)
\]
is well defined for every $x\in B$, and $x\mapsto \Delta f(x)$ is continuous whenever $f\in C^2(B)$.

\

\noindent
It follows from the assumption~\eqref{Eq:Def_F} that $F\in C^2(H)$ and that the gradient and Hessian of $F$ 
are given by 
\[
\nabla F(x)=x+\nabla U(x),
\qquad
\nabla^2F(x)=\mathrm{Id}+\nabla^2U(x),         \quad     x\in H , 
\]
where $\nabla U$ and $\nabla^2 U$ are understood in the Cameron--Martin sense introduced above. 
We shall denote by    
\[\mathcal{S}:=\left\{x\in H:\nabla F(x)=0\right\}   \]
the set of critical points of $F$. 
In finite dimensions, one can take $B=H$ and a natural object associated with  $F$ is the
Witten, or effective, potential
\begin{equation}
\label{Eq:Def_Witten}
\frac{\beta}{4}\|\nabla F(x)\|_H^2-\frac12\Delta F(x), \qquad x\in H, 
\end{equation}
depending on the parameter $\beta>0$.
In the present infinite-dimensional setting this expression is only formal.
Indeed, although $\nabla^2U(x)$ is trace class by Goodman's theorem, the identity
term in $ \nabla^2F=\mathrm{Id}+\nabla^2U$ has infinite trace. Thus $\Delta F$ is not defined as a finite quantity.
A natural renormalized expression, see Proposition~\ref{prop:GST} below, is obtained by removing the divergent
quadratic contribution coming from the quadratic part $F-U|_{H}=\frac12\|\cdot\|_H^2$. Formally, this
amounts to subtracting
$
\frac{\beta}{4}\|\nabla(F-U)\|_H^2
-\frac12\Delta(F-U)
$
from the expression~\eqref{Eq:Def_Witten}. Accordingly, we introduce for each $\beta>0$ the effective potential 
$V_\beta \in C(B)$ defined by 
\begin{equation}
\label{Eq:Def_V_beta}
    V_\beta(x) := \frac{\beta}{4} \|\nabla U(x)\|_H^2+ \frac{\beta}{2} \langle DU(x),x\rangle -   \frac 12\, \Delta U(x)     ,
    \qquad x\in B .
\end{equation}
We now collect the assumptions on the realized potential $F$ needed for the
spectral separation result. 
\begin{assumption}
\label{Ass_for_theorem_I}
The function $U$ is bounded from below and there exist constants $C,\tau>0$ such that
\begin{align}
\label{eq:expo_control}
    \left\|D U(x)\right\|_{B^*} \leq C\mathrm{exp}(\tau \|x\|_B^2) ,    \qquad    x\in B. 
\end{align}
Furthermore, we assume that for any $x\in \mathcal{S}$ with $\nabla^2 F(x)\geq 0$, there exists a constant $c>0$ such that $\nabla^2 F(x)\geq c$. Finally, the effective potential $V_\beta$ satisfies the uniform bound
\begin{align}
\label{Eq:Th_Assumption}
   \varliminf_{\beta}\inf_x \beta^{-1}V_\beta (x)>-\infty .
\end{align}
\end{assumption}
\noindent
For $\beta>0$, let  
$\gamma_\beta$ denote the push-forward of $\gamma$ under the scaling map $x\mapsto \beta^{-1/2} x$.
Under Assumption~\ref{Ass_for_theorem_I}, the function $U$ is bounded from below, hence  $Z_\beta =    \int_B e^{- \beta U }  d \gamma_\beta < \infty$.    We then define the probability measure 
\begin{align}
\label{Eq:def_mu}
\mathrm{d}\mu_{\beta}:=Z_\beta^{-1}e^{-\beta U}\mathrm{d}\gamma_\beta .
\end{align}
We shall use in the sequel the space $\operatorname{Lip}_b(B)$ of bounded, globally Lipschitz continuous functions on $B$. 
By the infinite-dimensional Rademacher theorem of Enchev and Stroock~\cite{EnchevStroock1993}, every $f\in \operatorname{Lip}_b(B)$ admits a Cameron--Martin gradient $\nabla f(x)$ for $\gamma_\beta$-almost every $x\in B$, for each fixed $\beta>0$. Since $\mu_\beta$ is equivalent to $\gamma_\beta$, the same holds $\mu_\beta$-almost everywhere. 
Moreover, denoting by $C_{\operatorname{Lip(H)}}(f)$ the Lipschitz constant of $f$
 along $H$-directions, it holds $\|\nabla f(x) \|_H     \le     C_{\operatorname{Lip(H)}}(f)$ at $x\in B$ where $\nabla f(x)$ exists. 

\

\noindent
In addition to the realized potential $F$, which determines the Gibbs measures
$\mu_\beta$ as above, we fix throughout a possibly unbounded selfadjoint    
operator $A$ on $H$, defined on a dense linear space $\operatorname{dom}(A) \subset H$. We assume that $A$ is uniformly positive and denote by
$a>0$ its lower bound. We also assume the following compatibility condition between $A$ and the abstract
Wiener realization. 
For $\ell\in B^*$, we denote by $\ell_H\in H$ its Cameron--Martin
representative, characterized by
\[ \ell(ih)=\langle \ell_H,h\rangle_H,
\qquad h\in H ,
\]
where  $i: H \to B$ is the Cameron--Martin embedding and set $B_H^* :=   \{     \ell_H:   \ell\in B^* \} \subset H$.
We shall assume throughout that the set of $A$-admissible Cameron-Martin directions
\[
\mathcal K_A
:=
\left\{
h \in    B_H^* \cap \operatorname{dom}(A): Ah \in  B_H^*
\right\}      \text{ is dense in } H. 
\]
We denote by $\mathcal F_A C_b^\infty(B)$ the set of cylindrical functions adapted to $A$, namely of functions $f: B \to\mathbb R$ such that
\[f(x)=\varphi\bigl(\ell_1(x),\ldots,\ell_n(x)\bigr)     \]
for some  $n\in\mathbb N$, $\varphi\in C_b^\infty(\mathbb R^n)$ and $\ell_1,\ldots,\ell_n\in B^*$
satisfying $(\ell_j)_H  \in \mathcal K_A $.
Because $\mathcal K_A$ is dense in $H$, the linear functionals
$\ell\in B^*$ with $\ell_H\in\mathcal K_A$ generate the full Gaussian first
chaos. By the Wiener chaos decomposition, the corresponding smooth cylindrical
functions are therefore dense in $L^2(\gamma_\beta)$.
Since $\mu_\beta$ is
equivalent to $\gamma_\beta$ and has bounded density with respect to
$\gamma_\beta$, a standard truncation argument shows that
$\mathcal F_A C_b^\infty(B)$ is also dense in $L^2(\mu_\beta)$, for every
$\beta>0$.
For a function $f\in \mathcal F_A C_b^\infty(B)$ the Cameron--Martin gradient is given by
\[
   \nabla f(x)
   =
   \sum_{j=1}^n
   \partial_j\varphi\bigl(\ell_1(x),\ldots,\ell_n(x)\bigr)\,(\ell_j)_H ,   \qquad x\in B. \]
 Since $(\ell_j)_H\in K_A$, the vector $\nabla f(x)$ belongs to
$\operatorname{dom}(A)$. The coefficients in the
above formula being bounded, it follows that $ x\mapsto \|\sqrt A \,\nabla f(x)\|_H$ is bounded on $B$.

\

\noindent
In view of the preceding discussion we may define on $\mathcal F_A C_b^\infty(B)$
the bilinear form
\[
\mathcal E_\beta[f,g]
:=
\frac1\beta
\int_B
\langle \sqrt A \, \nabla f , \sqrt A \nabla g\rangle_H \, d\mu_\beta ,
\]
and set for short  $\mathcal E_\beta[f]:=\mathcal E_\beta[f,f]$. For
$f\in \mathcal F_A C_b^\infty(B)$ and $x\in B$ we also define 
\begin{equation} \label{Eq:Def_L_Operator}
\mathcal L_\beta f(x)
:=
-\frac1\beta \operatorname{tr}    \left(A\nabla^2 f(x)\right)
+\langle A\nabla f(x),x\rangle
+\langle A\nabla f(x),\nabla U(x)\rangle_H .
\end{equation}
Each term on the right-hand side is well defined, and
$\mathcal L_\beta f\in L^2(\mu_\beta)$ for all sufficiently large $\beta$. Indeed, $A\nabla f(x)\in B_H^*$. Hence, through the identification of $B_H^*$
with $B^*$, the dual pairing $\langle A\nabla f(x),x\rangle$ is well defined for
every $x\in B$. Since this term is a finite linear combination of $B^*$-linear
functionals with bounded coefficients, it belongs to $L^2(\mu_\beta)$. 
Moreover, the growth assumption~\eqref{eq:expo_control} on $DU$, together with the lower boundedness of
$U$ and Fernique's theorem applied to the rescaled Gaussian measure
$\gamma_\beta$, implies that $\langle A\nabla f,\nabla U\rangle_H\in
L^2(B,\mu_\beta)$ for all sufficiently large $\beta$.
Finally, since $\nabla^2f(x)$ has finite rank and range contained
in $\operatorname{dom}(A)$, the operator $A\nabla^2f(x)$ is finite rank and the
trace term is bounded.

 An integration-by-parts formula gives, for all
$f,g\in\mathcal F_A C_b^\infty(B)$,
\[
   \mathcal E_\beta[f,g]
   =
   \int_B f\,\mathcal L_\beta g\,d\mu_\beta .
\]
In particular, $\mathcal L_\beta$ is symmetric and non-negative on
$\mathcal F_A C_b^\infty(B)$. Since $\mathcal F_A C_b^\infty(B)$ is dense in
$L^2(\mu_\beta)$, the nonnegative symmetric operator $\mathcal L_\beta$
admits a Friedrichs extension. Equivalently, the quadratic form
$\mathcal E_\beta$ is closable. We keep the notation $\mathcal E_\beta$ for its
closure and still denote by $\mathcal L_\beta$ the associated nonnegative
selfadjoint operator.

\

\noindent
The min--max values associated with the closed form $\mathcal E_\beta$ are defined by
\begin{equation}
\label{Eq:min_max_definition_eigenvalue}
   \lambda_\beta^{(j)}
   :=
   \inf_{\substack{V\subset \operatorname{dom}(\mathcal E_\beta)\\ \dim V=j+1}}
   \sup_{f\in V\setminus\{0\}}
   \frac{\mathcal E_\beta[f]}{\|f\|_{L^2(\mu_\beta)}^2},
   \qquad j\ge 0 .
\end{equation}
Equivalently, since $\mathcal F_A C_b^\infty(B)$ is a form core for
$\mathcal E_\beta$, the same values are obtained if the infimum is restricted to
subspaces $V\subset\mathcal F_A C_b^\infty(B)$. Since $\mathcal E_\beta[f]\ge0$
and $\mathcal E_\beta[1]=0$, we have $\lambda_\beta^{(0)}=0$.
\begin{theorem}
\label{Thm:Main_Abstract_Theorem}
Under Assumption \ref{Ass_for_theorem_I}, the number $k_0$ of local minima of $F$ is finite and non-zero. Moreover  
\[   \varliminf_{\beta\rightarrow \infty} \lambda_{\beta}^{(k_0)}  >0 , \]
and, for $j    =   1, \dots, k_0-1$, 
\begin{align*}
     \varlimsup_{\beta\rightarrow \infty}\frac{1}{\beta}\log \lambda_{\beta}^{(j)}<0.
\end{align*}
\end{theorem}

\noindent
In order to obtain a sharp asymptotic on the low--lying eigenvalues, we make the additional Assumption \ref{Ass_for_theorem_II}.
\begin{assumption}
\label{Ass_for_theorem_II}
The function $F$ has exactly two local minima $x_\pm$ and $\nabla^2 F$ is invertible for all $x\in \mathcal{S}$. 
We further assume that, for every critical point $x\in \mathcal S$, there exist constants
$C_1,C_2>0$ such that
\begin{align}
\label{eq:assump_comp}
    \langle y,\nabla^2 F(x)y\rangle_H\geq C_1 \|y\|_H^2-C_2\langle y,A^{-1}y\rangle_H.
\end{align}
\end{assumption}

\noindent
In the following we denote by $\overline{H}$ the completion of $H$ with respect to the inner product $\langle x,y\rangle_A:=\langle A^{-1}x,y\rangle$ and we write $\mathcal{T}$ for the set of critical points $x\in \mathcal{S}$ such that $\nabla^2 F(x)$ has exactly a single negative eigenvalue. We obtain by Assumption \eqref{eq:assump_comp} that the unbounded quadratic form 
   \begin{align*}
       q(y):=\langle y,\nabla^2 F(x)y\rangle_H
   \end{align*}
 defined on the subspace $H\subseteq  \overline{H}$ (equipped with the norm $\|\cdot \|_A$) is complete, with a corresponding self-adjoint operator $Q$ defined on $\overline{H}$ satisfying
 \begin{align}
 \label{Eq:Q_Operator}
    A^{-1}Q=\nabla^2 F(x)|_{\operatorname{dom}(Q)}.
 \end{align}
While the spectrum of $Q$ and $\nabla^2 F(x)$ is different in general, we note that the kernel of $Q$ is trivial and the number of negative eigenvalues of $Q$ agrees with the number of negative eigenvalues of $\nabla^2 F(x)$. In particular, we have for $x\in \mathcal{T}$ that there is, up to a sign, a unique eigenvector $\nu_x\in H$ of $Q$ with a negative eigenvalue $-\alpha_x<0$ and $\|\nu_x\|_A=1$. By Assumption \ref{Ass_for_theorem_II} it is clear that $Q$ is invertible, i.e. $0\notin \sigma(Q)$, and hence we find a $\delta>0$ such that
\begin{align*}
   \langle w, Qw\rangle_A \geq \delta \|w\|_A^2
\end{align*}
for all $w\in \overline{H}$ with $\langle \nu_x,w\rangle_A=0$. Defining the vector 
\begin{align}
    \label{Eq:def_of_eta_x}
    \eta_x:=A^{-1}\nu_x,
\end{align}
we therefore obtain, together with Assumption \eqref{eq:assump_comp}, for all $w\in H$ with $\langle \eta_x,w\rangle=0$
\begin{align*}
    \left\langle w, \nabla^2 F(x)w \right\rangle\geq \delta' \|w\|^2,
\end{align*}
for a suitable $\delta'>0$. In particular, the vector $\eta_x$ separates the level set of the quadratic form $[q\leq 0]$ into two connected components, or equivalently the orthogonal space $\langle \eta_x\rangle^{\perp}$ has a trivial intersection with $[q\leq 0]$. \\

\noindent
The preceding local picture motivates the following standard Morse-theoretic description of the critical sublevel set at the communication height. Let \(h\) be the infimum of the maximal value attained along paths connecting \(x_\pm\), and let \(\mathcal T_h\) denote the set of critical points \(x\in [F\le h]\) whose Hessian has exactly one negative eigenvalue. We use the corresponding decomposition of the critical sublevel set: the set \([F\le h]\setminus \mathcal T_h\) has two path-connected components, denoted by \(A_-\) and \(A_+\), containing respectively \(x_-\) and \(x_+\). Equivalently, the two components meet only through the saddle region \(\mathcal T_h\). In particular, after removing small balls around the points of \(\mathcal T_h\), the two components are separated by a positive distance in the Banach-space topology:
 \begin{equation}     \label{eq:dist}
 \operatorname{dist}_B \bigl(A_- \setminus B_r(\mathcal T_h), A_+ \setminus B_r(\mathcal T_h)\bigr)>0 . \end{equation} 
 Here \[ B_r(\mathcal T_h):=\bigcup_{y\in\mathcal T_h} B_r(y), \qquad B_r(y):=\{z\in B:\|z-y\|_B<r\}. \] Moreover, if \(x\in \mathcal T_h\), then the sign of \(\langle \eta_x,y-x\rangle\) distinguishes the two local branches of \([F\le h]\) near \(x\). More precisely, for \(R>0\) sufficiently small, all points \(y\in [F\le h]\cap B_R(x)\) with \[ \langle \eta_x,y-x\rangle<0 \] belong to one of the two components, denoted by \(A_{\sigma_1(x)}\), while all points \(y\in [F\le h]\cap B_R(x)\) with \[ \langle \eta_x,y-x\rangle>0 \] belong to the other component, denoted by \(A_{\sigma_2(x)}\).

\

\noindent
In order to state our second main Theorem \ref{Th:precise_Asy}, let us sort
\begin{align}
\label{Eq:enummeration_crit}
    \mathcal{T}_h=\{x_1,\dots ,x_{m_0},x_{m_0+1},\dots  ,x_m\},
\end{align}
such that $\sigma_1(x_j)\neq \sigma_2(x_j)$ for $j\leq m_0$ and $\sigma_1(x_j)=\sigma_2(x_j)$ otherwise, and let $\{-\alpha_1,\dots,-\alpha_m\}$ as well as $\{\eta_1,\dots,\eta_{m}\}$ be the corresponding enumeration of the eigenvalues $-\alpha_x$ and vectors $\eta_j:=A^{-1}\nu_j$ defined in \eqref{Eq:def_of_eta_x}. Then, using the rank one projection $P_{\eta,\eta'}(y):=\langle \eta',y\rangle \eta$, we define the Gaussian measures
\begin{align*}
    \mathrm{d}\gamma_{\pm}: & =C_\pm^{-1} e^{-\frac{1}{2}\langle x,\nabla^2  U(x_\pm)x\rangle}\mathrm{d}\gamma,\\
     \mathrm{d}\gamma_{j}: & =C_j^{-1} e^{- \frac{1}{2}\left\langle x,\left(\nabla^2  U(x_j)+2\alpha_j P_{\eta_j,\eta_j}\right)x\right\rangle }\mathrm{d}\gamma,
\end{align*}
with the normalization constants $C_\pm$ and $C_j$ given as
\begin{align}
\nonumber
    C_\pm: & =\mathrm{det}\Big[\mathbf{1}+\nabla^2  U(x_\pm)\Big]^{-\frac{1}{2}}=\mathrm{det}\Big[\nabla^2  F(x_\pm)\Big]^{-\frac{1}{2}},\\
    \label{Eq:def_C_j_const}
    C_j: & = \mathrm{det}\Big[\mathbf{1}+\nabla^2  U(x_j)+2\alpha_j P_{\eta_j,\eta_j}\Big]^{-\frac{1}{2}}= \left|\mathrm{det}\Big[\nabla^2  F(x_j)\Big]\right|^{-\frac{1}{2}}.
\end{align}
The identity in \eqref{Eq:def_C_j_const} follows easily from the fact that
\begin{align*}
\mathbf{1}+\nabla^2  U(x_j)+2\alpha_j P_{\eta_j,\eta_j}=\nabla^2  F(x_j)R_j ,    
\end{align*}
where $R_j:=I-2P_{\nu_j,\eta_j}$ is an isometry with respect to the
$\|\cdot\|_A$ metric. Furthermore, we define the re-scaled versions
\begin{align*}
    \mathrm{d}\gamma_{\beta,\pm}: & =C_\pm^{-1} e^{-\frac{\beta}{2} \langle x,\nabla^2   U(x_\pm)x\rangle}\mathrm{d}\gamma_{\beta},\\
     \mathrm{d}\gamma_{\beta,j}: & =C_j^{-1} e^{- \frac{\beta}{2} \left\langle x,\left(\nabla^2  U(x_j)+2\alpha_j P_{\eta_j}\right)x\right\rangle }\mathrm{d}\gamma_{\beta}.
\end{align*}

\begin{theorem}
\label{Th:precise_Asy}
Assume that Assumptions~1 and~2 hold, that $\eta_x\in\mathcal K_A$ for every
$x\in \mathcal T_h$, and that $U\in C^3(B)$. Then the following asymptotics hold  as $\beta\to \infty$.

If $F(x_-)=F(x_+)$, then
\begin{align}
\label{Eq:precise_Asy_1}
   \lambda_\beta^{(1)}
   =
      \sum_{j=1}^{m_0}
   \frac{C_-+C_+}{2\pi C_-C_+}\,
   C_j\alpha_j\,
   e^{-\beta\left(h-F(x_+)\right)}  \left(1+O  ( \beta^{-1/2} )\right)  .
\end{align}

If $F(x_+)>F(x_-)$, then
\begin{align}
\label{Eq:precise_Asy_2}
   \lambda_\beta^{(1)}
   =
     \sum_{j=1}^{m_0}
   \frac{C_j\alpha_j}{2\pi C_+}\,
   e^{-\beta\left(h-F(x_+)\right)}    \left(1+O(\beta^{-1/2})\right).      \end{align}
\end{theorem}

\section{Proof of the spectral separation}   \label{sec3}

\noindent
It is the content of this section to verify Theorem \ref{Thm:Main_Abstract_Theorem}, which claims that, depending on the number of minima $k_0$ of $F$, the eigenvalue $\lambda^{(k_0)}_\beta$ is bounded away from zero by a constant $\delta>0$ uniformly in $\beta\rightarrow \infty$. From the min-max definition of $\lambda^{(k_0)}_\beta$, it is clear that it is enough to find elements $\Phi_\beta^{(1)},\dots, \Phi_\beta^{(k_0)}$, such that
\begin{align}
\label{Eq:lower_bound_global}
    \mathcal{E}_\beta[\Psi]\geq \delta \|\Psi\|_{L^2(\mu_\beta)}^2,
\end{align}
for all $\Psi$ orthogonal to $\Phi_\beta^{(1)},\dots, \Phi_\beta^{(k_0)}$. In the following, we will first verify the lower bound in \eqref{Eq:lower_bound_global} separately for functions $\Psi$ supported in a neighborhood of a critical point $x\in \mathcal{S}$, where we distinguish between minima and other critical points, and functions having a (certain) positive distance from the critical points $\mathcal{S}$. In our final Subsection \ref{Subsec:glueing} we verify Theorem \ref{Thm:Main_Abstract_Theorem} by employing a localization technique and gluing together the results from the different regions.

In order to decompose $B$ into the different relevant regions, we note that, as a consequence of Assumption \ref{Ass_for_theorem_I}, we can split the critical points as $\mathcal{S}=\mathcal{S}_0\cup \mathcal{S}_1$ with
\begin{itemize}
    \item for all $x\in \mathcal{S}_0$ there exist $c_x,r_x>0$, with
    \begin{align}
    \label{Eq:Open_Convex}
      \nabla^2 F(y)\geq c_x  >0
    \end{align}
for all $y\in B_{r_x}(x)\cap H$,\\
    \item for all $x\in \mathcal{S}_1$ there exists $c_x,r_x>0$ and a direction $\zeta_x\in H$, with 
    \begin{align}
        \label{Eq:Open_Saddle}
        \langle \zeta_x, \nabla^2 F(y)\,  \zeta_x\rangle\leq -c_x<0
    \end{align}
   for all $y\in B_{r_x}(x)\cap H$.
\end{itemize}
In the following Subsection \ref{Subsec:Analysis away from the critical points} we investigate functions with a support away from the critical points $\mathcal{S}$, and subsequently we investigate functions supported in one of the regions $B_{r_x}(x)$ in Subsection \ref{Subsec:Analysis around the critical points}.

\subsection{Analysis away from the critical points}
\label{Subsec:Analysis away from the critical points}
We first recall three standard tools in Propositions~\ref{prop:GST},
\ref{Prop:2} and~\ref{Prop:4}: the ground-state transformation, the NGS bound
for Schr\"odinger-type quadratic forms, and a large deviation upper bound for
exponential integrals. The first proposition is elementary and shows how to transform the underlying
base measure of a Dirichlet form, at the cost of introducing the additional
potential $V_\beta$.
\begin{proposition}[Ground state transformation]
\label{prop:GST}
Let $f\in \operatorname{Lip}_b(B)$
and set   \\ $g=Z_\beta^{-1/2} f e^{-\beta U/2}  $.    Then
\begin{equation}
\label{Eq:Schroedinger_Op}
   \frac1\beta\int_B \|\nabla f\|_H^2\,d\mu_\beta
   =
   \int_B
   \left[
      \frac1\beta\|\nabla g\|_H^2
      +V_\beta g^2
   \right]\,d\gamma_\beta ,
   \end{equation}
where $V_\beta$ is defined in \eqref{Eq:Def_V_beta}.
\end{proposition}
\noindent
The next proposition is a localized form of the Nelson--Glimm--Segal
bound, or NGS bound, see \cite[p.~159]{Simon1974}. It follows from the
Gaussian logarithmic Sobolev inequality
\[
   \frac1\beta
   \int_B \|\nabla f\|_H^2\,d\gamma_\beta
   \ge
   \frac12\,\operatorname{Ent}_{\gamma_\beta}(f^2),
\]
and the entropy variational formula, see also  \cite[Theorem~7]{Gross1975}.
\begin{proposition} [NGS bound]
\label{Prop:2}
Let $g\in \mathrm{Lip}_b(B)$ with $\|g\|_{L^2(d\gamma_\beta)} =1$. Then 
\begin{align}
\label{Eq:Schroedinger_op_lower_bound}
  \int_B
   \left[
      \frac1\beta\|\nabla g\|_H^2
      +V_\beta g^2
   \right]\,d\gamma_\beta
   \ge
   -\frac12\log\int_{\operatorname{supp}(g)} e^{-2V_\beta}\,d\gamma_\beta,
\end{align}
where $V_\beta$ is defined in \eqref{Eq:Def_V_beta}.
\end{proposition}
\noindent
The exponential integral on the right-hand side of
\eqref{Eq:Schroedinger_op_lower_bound} can be controlled by large-deviation
estimates. More precisely, we shall use the following variant of Varadhan's
lemma~\cite{Varadhan1966}, \cite[Section~4.3]{DemboZeitouni1998}.

\begin{proposition}[Varadhan upper bound]
\label{Prop:4}
For each $\beta>0$, let $\sigma_\beta$ be a probability measure on a separable
Banach space $B$,  let $h_\beta:B\to(-\infty,+\infty]$ be measurable and  let
$J:B\to[0,\infty]$ have compact sublevel sets. 
Assume that $ \widetilde h_\alpha:=\inf_{\beta\ge\alpha} h_\beta $
is lower semicontinuous for every $\alpha>0$
and that
\begin{itemize}
\item[(i)]  $
   \varlimsup_{\beta\to\infty}
   \frac1\beta\log\sigma_\beta(C)
   \le
   -\inf_{x\in C}J(x) $   for every closed set $C\subset B$;
\item[(ii)] $h_\beta\to h$ uniformly on compact sets;
\item[(iii)]
$   \varliminf_{\beta\to\infty}\inf_{x\in B}h_\beta(x)>-\infty  $.
\end{itemize}
Then, for every closed set $C\subset B$,
\[
   \varlimsup_{\beta\to\infty}
   \frac1\beta
   \log\int_C e^{-\beta h_\beta}\,d\sigma_\beta
   \le
   -\inf_{x\in C}\bigl[h(x)+J(x)\bigr].
\]
\end{proposition}

\begin{remark}
\label{Remark:Trace}
In the application below we take $\sigma_\beta=\gamma_\beta$ and
\[
   h_\beta
   :=
   2\beta^{-1}V_\beta
   =
   \frac12\|\nabla U(x)\|_H^2
   +\langle DU(x),x\rangle
   -\frac1\beta\,\Delta U(x).
\]
Assumption~(i) then follows from the Gaussian large-deviation principle for
$\gamma_\beta$, with rate function
\[
   J(x)
   =
   \begin{cases}
      \frac12\|x\|_H^2, & x\in H,\\
      +\infty, & x\in B\setminus H.
   \end{cases}
\]
Moreover, Assumption~(ii) holds with
\[
   h(x)
   =
   \frac12\|\nabla U(x)\|_H^2
   +\langle DU(x),x\rangle,
\]
because $\Delta U$ is continuous and hence bounded on compact subsets of $B$.
Assumption~(iii) is precisely the uniform lower bound~\eqref{Eq:Th_Assumption} on the effective
potential $V_\beta$ in Assumption~\ref{Ass_for_theorem_I}.
Finally, the lower semicontinuity assumption on $\widetilde h_\alpha$ is automatic
in this situation. Indeed, if we write
\[
   h_\beta=h-\frac1\beta q,
   \qquad q:=\Delta U,
\]
then $h$ and $q$ are continuous and
\[
   \widetilde h_\alpha
   =
   \inf_{\beta\ge\alpha}
   \left(h-\frac1\beta q\right)
   =
   \min\left\{
      h,\,
      h-\frac1\alpha q
   \right\}.
\]
Thus $\widetilde h_\alpha$ is continuous, and in particular lower semicontinuous.
\end{remark}
\noindent
With the auxiliary results Proposition \ref{prop:GST}, Proposition \ref{Prop:2} and Proposition \ref{Prop:4} at hand, we are in a position to verify Theorem \ref{Eq:Main_away_from_crit}, which is the main result of this subsection.
\begin{theorem}
\label{Eq:Main_away_from_crit}
    Let $U\in C^2(B)$ satisfy \eqref{Eq:Th_Assumption}. Then the set $\mathcal{S}:=\left\{x\in H:\nabla F=0\right\}$ is compact in $B$, and for all $r>0$ there exist constants $\lambda,\beta_0>0$, such that
   \begin{align}
   \label{Eq:Energy_Lower_Bound_Away_From_Crit}
     \frac{1}{\beta} \int_B \|\nabla \Psi\|_H^2\mathrm{d}\mu_\beta\geq \lambda \beta \|\Psi\|^2_{L^2(\mu_\beta)}
   \end{align}
    for all $\beta\geq \beta_0$ and $\Psi\in \mathrm{Lip}_b(B)$ with $\mathrm{dist}_B\! \left(\mathrm{supp}(\Psi),\mathcal{S}\right)\geq r$.
\end{theorem}
\begin{proof}
We note that $h(x)+\frac{1}{2}\|x\|^2_H$ with $h(x):=\frac{1}{2}\|\nabla U\|_H^2 +  \langle \nabla U, \cdot \rangle$ is identical to $\frac{1}{2}\left\|\nabla F(x)\right\|_H^2$. Hence,
    \begin{align*}
        \mathcal{S}=\left\{x\in H:h(x)+\frac{1}{2}\|x\|^2_H=0\right\}.
    \end{align*}
    By Assumption \eqref{Eq:Th_Assumption}, it is clear that $h(x)\geq -c$, and therefore
    \begin{align}
        \label{Eq:h_J_lower_bound}
        h(x)+\frac{1}{2}\|x\|^2_H\geq 1
    \end{align}
for $x\in H\setminus K_{R}$ and $R$ large enough, where denote again $K_R:=\{x:\frac{1}{2}\|x\|^2_H\leq R\}$. In particular, $\mathcal{S}\subseteq K_R$ and we see that $\mathcal{S}$ is pre-compact in $B$. Equipping $H$ with the $B$-norm, the map $ \nabla F(x):H\longrightarrow H$ remains continuous, and hence $\mathcal{S}$ is compact in $B$. 

   Regarding the proof of \eqref{Eq:Energy_Lower_Bound_Away_From_Crit} we define the closed set
   \begin{align*}
     C:=\{x\in B: \mathrm{dist}_B(x,\mathcal{S})\geq r\}  
   \end{align*}
and use \eqref{Eq:h_J_lower_bound} again
    \begin{align*}
        \inf_{x\in C}[h(x)+\frac{1}{2}\|x\|^2_H]\geq \min\{1,\inf_{x\in C\cap K_R}[h(x)+\frac{1}{2}\|x\|^2_H]\}=\min\{1,h(x_*)+\frac{1}{2}\|x_*\|^2_H\},
    \end{align*}
  for some $x_*\in C\cap K_R$, where we utilize that $C\cap K_R$ is compact. Using again
  \begin{align*}
      h(x)+\frac{1}{2}\|x\|^2_H= \left\|\nabla F(x)\right\|_H^2\geq 0,
  \end{align*}
we have $\mathcal{S}=\left\{x\in H:h(x)+\frac{1}{2}\|x\|^2_H\leq 0\right\}$. Since $C$ is disjoint to $\mathcal{S}$, we consequently obtain $h(x_*)+\frac{1}{2}\|x_*\|^2_H>0$ and hence $ \inf_{x\in C}[h(x)+\frac{1}{2}\|x\|^2_H]>0$. Together with Proposition \ref{Prop:2} and Proposition \ref{Prop:4} this concludes the proof.
\end{proof}

\subsection{Analysis around the critical points}
\label{Subsec:Analysis around the critical points}
In the first result of this Subsection, we investigate functions $\Psi$ that are supported in a neighborhood of a local minimum $x$. Notably, the lower bound \eqref{Eq:lower_bound_conditional} holds only for elements that are orthogonal to the function $\Phi_\beta$ constructed below.
\begin{proposition}
\label{Prop:Min}
    Let $x \in \mathcal S_0$ and $B_{r_x}(x)$ be as in \eqref{Eq:Open_Convex}. Then, for all $0<r<r_x$, there exists a function $\Phi_\beta$ with $\|\Phi_\beta\|_{L^2(\mu_\beta)}=1$ and $\mathrm{supp}(\Phi_\beta)\subseteq B_{r}(x)$ and constants $C, \beta_0>0$
     such that for $\beta>\beta_0$
    \begin{align}
    \label{Eq:lower_bound_conditional}
         \frac{1}{\beta} \int_B \|\nabla \Psi\|_H^2 \, d\mu_\beta   \ge C\|\Psi\|^2_{L^2(\mu_\beta)}
\end{align}
for all $\Psi\perp_{L^2(\mu_\beta)} \Phi_\beta$ with $\mathrm{supp}(\Psi)\subseteq B_{r_x}(x)$. Furthermore, there exists $C,\kappa>0$ such that
\begin{align}
\label{Eq:generic_exponential_control}
     \frac{1}{\beta} \int_B \|\nabla \Phi_\beta\|_H^2 d\mu_\beta \leq C e^{-\beta \kappa}.
\end{align}
\end{proposition}
\begin{proof}
We define $\tilde U: B \to (-\infty, + \infty]$ by setting $\tilde U= U$ on $B_{r_x}(x)$ and $\tilde U =+\infty$ otherwise. 
We observe that the functional $h\mapsto \frac{1-c_x}{2} \|h\|_H^2 + \tilde U(h)$ is convex. By~\cite[Theorem 11.5.4]{Ustunel2015}
it follows that for every $\Psi \in {\rm{Lip}_b(B)}$ with $\int_B \Psi e^{-\beta \tilde U} d\gamma_\beta =0$
\[      \frac{1}{\beta}\int_B \|\nabla \Psi\|_H^2  e^{-\beta \tilde U} d\gamma_\beta     \ge  c_x  \int_B |\Psi|^2 e^{-\beta \tilde U} d\gamma_\beta .   \]
Fix $0<r<r_x$ and set $\tilde \Phi_\beta(y):=\left(\int_{B_{r_x}(x)}\mathrm{d}\mu_\beta\right)^{-\frac{1}{2}}$ for $y\in B_{r_x}(x)$ and $\tilde \Phi_\beta(y):=0$ otherwise) and
$\Phi_\beta(y):=\|\chi(\|\cdot -x\|_B)\|_{L^2(\mu_\beta)}^{-1}\chi(\|y-x\|_B)$, where 
$\chi$ is a smooth function with $\chi(t)=1$ in case $|t|\leq r/2$ and $\chi(t)=0$ for $|t|>r$. Using the strict convexity of $F$ on $B_{r_x}(x)$, with $x$ being the global minimum of $F$ within $B_{r_x}(x)$ we obtain by the Laplace asymptotic for the Gaussian measure $\gamma_\beta$ for suitable constants $C,\kappa>0$
\begin{align}
\label{Eq:laplace_asy_around_local_minima}
    \left(\int_{B_{r_x}(x)}\mathrm{d}\mu_\beta\right)^{-\frac{1}{2}}\int_{B_{r_x}(x)\setminus B_{\frac{r}{2}}(x)}\mathrm{d}\mu_\beta\leq Ce^{-\kappa\beta},
\end{align}
and in particular
\begin{align*}
   \|\tilde \Phi_\beta -\Phi_\beta\|_{L^2(\mu_\beta)}\leq \frac{1}{\sqrt{2}} 
\end{align*}
 for $\beta$ large enough. Hence we obtain for all $\Psi$ orthogonal to $\Phi_\beta$ that the orthogonal projection $\pi_\beta \Psi$ away from $\tilde\Phi_\beta$ satisfies $\|\pi_\beta\Psi\|_{L^2(\mu_\beta)}^2\geq \frac{1}{2}\|\Psi\|_{L^2(\mu_\beta)}^2$, and in particular
\begin{align*}
 \frac{1}{\beta}\int_B \|\nabla \Psi\|_H^2  e^{-\beta U} d\gamma_\beta \geq c_x \left\|\pi_\beta\Psi\right\|^2_{L^2(\mu_\beta)}\geq \frac{c_x}{2}\|\Psi\|^2_{L^2(\mu_\beta)}.
\end{align*}
Regarding \eqref{Eq:generic_exponential_control}, we note that $\|\chi(\|\cdot -x\|_B)\|_{L^2(\mu_\beta)}^2=\left(1+O_{\beta\rightarrow \infty}\!\left(e^{-\kappa\beta}\right)\right)\int_{B_{r_x}(x)}\mathrm{d}\mu_\beta$ by \eqref{Eq:laplace_asy_around_local_minima}. Furthermore, $\nabla \chi(\|\cdot -x\|_B)$ is an almost everywhere bounded function supported in $B_{r_x}(x)\setminus B_{\frac{r}{2}}(x)$, hence we obtain for suitable constants $C_1,C_2>0$
\begin{align*}
\int_B \|\nabla \Phi_\beta\|_H^2 d\mu_\beta & =\left(1+O_{\beta\rightarrow \infty}\!\left(e^{-\kappa\beta}\right)\right)\left(\int_{B_{r_x}(x)}\mathrm{d}\mu_\beta\right)^{-\frac{1}{2}}\int \|\nabla \chi(\|\cdot -x\|_B)\|^2\mathrm{d}\mu_\beta\\
& \leq C_1 \left(\int_{B_{r_x}(x)}\mathrm{d}\mu_\beta\right)^{-\frac{1}{2}}\int_{B_{r_x}(x)\setminus B_{\frac{r}{2}}(x)} \mathrm{d}\mu_\beta \leq C_2 e^{-\kappa\beta}.
\end{align*}
\end{proof}
\noindent
For critical points $x\in \mathcal{S}$ that are not local minima, we obtain an unconditional lower bound in the following Proposition \ref{Prop:Saddle}.
\begin{proposition}
\label{Prop:Saddle}
    Let $x,c_x,B_{r_x}(x)$ and $\eta:=\zeta_x$ be as in \eqref{Eq:Open_Saddle}. Then, there exists a constant $C>0$ such that for all $\Psi\in \mathrm{Lip}_b(B)$ with $\mathrm{supp}(\Psi)\subseteq B_{r_x}(x)$
    \begin{align*}
         \frac{1}{\beta}\int_B \|\nabla \Psi\|_H^2  d\mu_\beta    \geq C\|\Psi\|^2_{L^2(\mu_\beta)}.
    \end{align*}
\end{proposition}
\begin{proof}
   We define $\Phi:=Z_{\beta}^{-\frac{1}{2}}e^{-\frac{\beta}{2}U}\Psi$ and compute with the Paley-Wiener integral $\mathcal{I}$
    \begin{align}
    \nonumber 
         &  \frac{1}{\beta Z_\beta}\int \|\nabla \Psi\|^2_{H}e^{-\beta U}\mathrm{d}\gamma_{\beta} \geq \frac{1}{\beta\|\eta\|_H^2 Z_\beta }\int \|\partial_\eta  \Psi\|^2_{H}e^{-\beta U}\mathrm{d}\gamma_{\beta}\\
         \nonumber 
         & \ \ =\frac{1}{\beta \|\eta\|_H^2}\left\|\left(\partial_\eta +\frac{\beta}{2}\partial_{\eta}U\right)\Phi\right\|_{L^2(\mathrm{d}\gamma_{\beta})}^2\\
         \nonumber 
          & \ \  =\frac{1}{\beta \|\eta\|_H^2}\left\|\left(\partial_\eta +\frac{\beta}{2}\partial_{\eta}F-\frac{\beta}{2} \mathcal{I}(\eta)\right)\Phi\right\|_{L^2(\mathrm{d}\gamma_{\beta})}^2\\
         \nonumber
         & \ \  =\frac{1}{\beta \|\eta\|_H^2}\left\|\left(\partial_\eta -\frac{\beta}{2} \mathcal{I}(\eta)\right)\Phi\right\|_{L^2(\mathrm{d}\gamma_{\beta})}^2+\frac{1}{\beta \|\eta\|_H^2}\left\|\frac{\beta}{2}\partial_{\eta}(F)\Phi\right\|_{L^2(\mathrm{d}\gamma_{\beta})}^2  \\
         \nonumber 
         & \ \  \ \ \ \   - \frac{1}{\|\eta\|_H^2}\int \beta \mathcal{I}(\eta)|\Phi|^2 \partial_{\eta}\mathcal{F}\mathrm{d}\gamma_\beta+\frac{1}{2\|\eta\|_H^2}\int   \partial_{\eta}\big(|\Phi|^2\big)\partial_{\eta}\mathcal{F}\mathrm{d}\gamma_\beta\\
             \nonumber 
         & \ \ \geq - \frac{1}{\|\eta\|_H^2}\int \beta \mathcal{I}(\eta)|\Phi|^2 \partial_{\eta}\mathcal{F}\mathrm{d}\gamma_\beta+\frac{1}{2\|\eta\|_H^2}\int   \partial_{\eta}\big(|\Phi|^2\big)\partial_{\eta}\mathcal{F}\mathrm{d}\gamma_\beta\\
         \label{Eq:Partial_Integration}
         & \ \ = -\frac{1}{2\|\eta\|_H^2}\int  |\Phi|^2 \partial_{\eta}^2\mathcal{F}\mathrm{d}\gamma_\beta \geq \frac{c_x}{2\|\eta\|_H^2}\left\|\Phi\right\|^2_{L^2\! \left(\mathrm{d}\gamma_{\beta}\right)}= \frac{c_x}{2\|\eta\|_H^2}\left\|\Psi\right\|^2_{L^2\! \left(\mathrm{d}\mu_{\beta}\right)},
    \end{align}
    where we have used partial integration with respect to the Gaussian measure $\gamma_\beta$ in the first identity in Eq.~(\ref{Eq:Partial_Integration}). 
\end{proof}

\subsection{Gluing of the regions}
\label{Subsec:glueing}
\begin{proof}[Proof of Theorem \ref{Thm:Main_Abstract_Theorem}] 
    Regarding the finiteness of minima, let $U_x:=B_{r_x}(x)$ be as in \eqref{Eq:Open_Convex} and \eqref{Eq:Open_Saddle}. We note that $\mathcal{S}_0=\mathcal{S}\cap \{x\in B:\nabla^2 F(x)\geq 0\}$ is compact in the $B$-metric. We have $\mathcal{S}_0\subseteq \bigcup_{x\in \mathcal{S}_0}U_x$ and consequently there exists a finite cover $\mathcal{S}_0\subseteq \bigcup_{j=1}^{k_0}U_{x_j}$ with $x_j\in \mathcal{S}_0$. However, $U_x\cap \mathcal{S}=\{x\}$ by the strict convexity of $F$ on $U_x$ and hence
    \begin{align*}
        \mathcal{S}_0=\left(\bigcup_{j=1}^{k_0}U_{x_j}\right)\cap \mathcal{S}_0=\left\{x_1,\dots ,x_{k_0}\right\}.
    \end{align*}
    We further observe that $\mathcal{S}_1=\mathcal{S}\setminus \{x\in B:\nabla^2 F(x)> 0\}$ is compact as well and choose a finite cover $\mathcal{S}_1\subseteq \bigcup_{j=1}^n U_{x_{k_0+j}}$, where $x_{{k_0}+j}\in \mathcal{S}_1$.
    
    We now want to decompose a given $\Psi\in \mathrm{Lip}_b(B)$ according to the cover
    \begin{align*}
        B=\bigcup_{j=1}^{n+k_0}U_{x_j}\cup \{y\in B: \mathrm{dist}_B(y,\mathcal{S})\geq  \delta\},
    \end{align*}
   where we choose $\delta>0$ small enough such that $B_{2\delta}(x_j)\subseteq U_{x_j}$ and such that the sets $B_{2\delta}(x_j)$ are disjoint for $j\in \{1,\dots, n+k_0\}$.  In the following let $f_\pm:[0,\infty)\longrightarrow [0,1]$ be a smooth quadratic partition of unity, i.e. $f_-^2+f_+^2=1$ with $\mathrm{supp}(f_-)\subseteq [0,2\delta)$ and $\mathrm{supp}(f_+)\subseteq (\delta,\infty)$. We then define for $\Psi\in \mathrm{Lip}_b(B)$
    \begin{align*}
        \Psi_0 (x) : & = g_0(x)\Psi(x) := \prod_{i=1}^{n+{k_0}}f_{+}(\|x - x_i\|_B)\Psi(x),\\
        \Psi_j (x) : & = g_j(x) \Psi(x) := f_-(\|x - x_j\|_B)\Psi(x)
    \end{align*}
   for $j\in \{1,\dots , n+{k_0}\}$ and note that $\sum_{j=0}^{n+{k_0}}g_j^2=1$. Clearly, $x\mapsto \|x-x_j\|_B$ is Lipschitz continuous (with constant $C=1$) and therefore the functions $\Psi_j$ are in $\mathrm{Lip}_b(B)$. Using $A\geq a>0$ we have
    \begin{align*}
        \mathcal{E}_\beta[\Psi]\geq \frac{a}{\beta}\int_B \|\nabla \Psi\|_H^2\mathrm{d}\mu_\beta=\frac{a}{\beta}\sum_{j=0}^{n+{k_0}}\int_B \|\nabla \Psi_j\|_H^2\mathrm{d}\mu_\beta-\frac{a}{2\beta}\mathcal{R}[\Psi],
    \end{align*}
    where we define $r_j(x):=\left\|\nabla g_j\right\|^2_H$ and
    \begin{align*}
        \mathcal{R}[\Psi]:=\sum_{j=0}^{n+{k_0}} \int r_j |\Psi|^2\, \mathrm{d}\mu_\beta.
    \end{align*}
    Using that the Lipschitz constant of $g_j$ is bounded, it is clear that the norm of $Dg_j\in B^*$ is uniformly bounded and hence $\|r_j\|_\infty<\infty$ as well. In the following let $\Phi_j$ be the functions from Proposition \ref{Prop:Min} for the local minima $x_j$ and radius $r:=\delta$ and let $\Psi$ be a function such that $\Psi\perp_{L^2(\mu_\beta)} \Phi_j=0$ for $j\in \{1,\dots ,{k_0}\}$. Since $g_j=1$ on the support of $\Phi_j$ we have
    \begin{align*}
        \langle \Psi_j,\Phi_j\rangle_{L^2(\mu_\beta)}= \langle \Psi,\Phi_j\rangle_{L^2(\mu_\beta)}=0,
    \end{align*}
    i.e. $\Psi_j\perp_{L^2(\mu_\beta)} \Phi_j$. Combining Theorem \ref{Eq:Main_away_from_crit}, Proposition \ref{Prop:Min} and Proposition \ref{Prop:Saddle} we obtain for some $c>0$
    \begin{align*}
        \mathcal{E}_\beta[\Psi] & \geq c\sum_{j=0}^{n+{k_0}}\left\|\Psi_j\right\|_{L^2(\mu_\beta)}^2  -\frac{\sum_{j=0}^{n+k_0}\|r_j\|_\infty}{2\beta}  \left\|\Psi\right\|_{L^2(\mu_\beta)}^2\\
        &= \left(c-\frac{\sum_{j=0}^{n+{k_0}}\|r_j\|_\infty}{2\beta}\right)\left\|\Psi\right\|_{L^2(\mu_\beta)}^2\\
        &\geq \frac{c}{2}\left\|\Psi\right\|_{L^2(\mu_\beta)}^2,
    \end{align*}
    for $\beta$ large enough. By the minmax principle, we conclude $\lambda_{\beta}^{(k_0)}\geq \frac{c}{2}>0$.

    Finally, in order to show that the first $k_0$ eigenvalues $\lambda^{(0)}_\beta,\dots ,\lambda^{(k_0-1)}_\beta$ are exponentially small we define the $k_0$ dimensional vector space 
    \begin{align*}
     \mathcal{V}:=\mathrm{span}\left\langle \Phi_j:j\in \{1,\dots ,{k_0}\}\right\rangle  ,
    \end{align*}
    By the definition of $\lambda_{\beta}^{(k_0)}$ in \eqref{Eq:min_max_definition_eigenvalue} and the upper bound in \eqref{Eq:generic_exponential_control} we obtain for $j\in  \{0,\dots ,{k_0}-1\}$ and suitable constants $C,\kappa>0$
    \begin{align*}
        \lambda^{(j)}_\beta\leq \lambda^{(k_0-1)}_\beta\leq \sup_{f\in \mathcal{V}\setminus \{0\}}\frac{\mathcal{E}_{\beta}[f]}{\|f\|^2_{L^2(\mathrm{d}\mu_{\beta})}}\leq Ce^{-\beta \kappa}.
    \end{align*}
\end{proof}

\section{Exact Asymptotics}     \label{sec4}

We assume in this Section that both Assumption \ref{Ass_for_theorem_I} and Assumption \ref{Ass_for_theorem_II} hold, as well as $\eta_x\in B^*$ for all critical points $\mathcal{T}_h$ and w.l.o.g. that $F(x_-)=0$. We further recall the definition of $A_\pm$ as the connected component of $[F\leq h]\setminus \mathcal{T}_h$ containing $x_\pm$, and introduce for $x\in \mathcal{T}_h$ the half-spaces
\begin{align*}
   H_1^\epsilon(x): & =[\langle \xi,\cdot \rangle  <  \langle \xi,x\rangle - \epsilon],\\ 
   H_2^\epsilon(x): & =[\langle \xi,\cdot \rangle  >  \langle \xi,x\rangle + \epsilon].
\end{align*}
As is discussed below \eqref{eq:dist}, locally, the points in $H_i^0(x)\cap [F\leq h]$ are all in the same connected component, which is denoted by $A_{\sigma_i(x)}$.

In order to verify Theorem \ref{Th:precise_Asy}, we will construct a trial state that is piece-wise defined on the regions $B_\delta \big(A_{\pm}\setminus B_r(x)\big)$ and $B_r(x)$ respectively for $x\in \mathcal{T}_h$. In order for this piece-wise definition to be consistent, it will be imperative to have a precise understanding of the intersections $B_r(x)\cap B_\delta \big(A_{\pm}\setminus B_r(x)\big)$, which is the content of the subsequent Proposition \ref{Cor:epsilon_half_space}.
\begin{proposition}
\label{Cor:epsilon_half_space}
    Given $x\in \mathcal{T}_h$ let $\eta :=\eta_x$ be as in \eqref{Eq:def_of_eta_x}. For $\epsilon,r,\delta>0$ small enough, $\sigma\in \{\pm\}$ and $i\in \{1,2\}$ 
    \begin{align*}
        H_i^0(x)\cap B_r(x)\cap B_\delta \big(A_{\sigma}\setminus B_r(x)\big)=H_i^{\epsilon}(x)\cap B_r(x)\cap B_\delta \big(A_{\sigma}\setminus B_r(x)\big),
    \end{align*}
 and $ H_i^0(x)\cap B_r(x)\cap B_\delta \big(A_{\sigma}\setminus B_r(x)\big)=\emptyset$ in case $\sigma\neq \sigma_i(x,\eta)$.
\end{proposition}
\begin{proof}
  We first show that, depending on $r,\delta>0$ small enough, we can choose $\epsilon_0>0$ such that for $0<\epsilon<\epsilon_0$
\begin{align}
\label{Eq:boost_by_epsilon}
    H_i^0(x)\cap B_r(x)\cap B_\delta \big(A_{\sigma}\setminus B_r(x)\big)=H_i^\epsilon(x)\cap B_r(x)\cap B_\delta \big(A_{\sigma}\setminus B_r(x)\big).
\end{align}
For this purpose we note that, in case $r+\delta>0$ is small enough, there exist  $c_1,c_2>0$, such that for any element in $y\in [F\leq h]\cap B_{r+\delta}(x)$
\begin{align*}
    - c_1|\langle \eta,y-x\rangle_H|^2+c_2 \|y-x\|_H^2\leq F(y)-F(x)\leq 0,
\end{align*}
and in particular $g(y):=|\langle \eta,y-x\rangle|$ satisfies
\begin{align*}
    \inf \Big\{|g(y)|:y\in \left(A_{\sigma}\setminus B_r(x)\right)\cap B_{r+\delta}(x)\Big\}\geq \sqrt{c_2/c_1}r>0.
\end{align*}
Since $\eta\in B^*\subseteq H$, it is clear that $g$ is a continuous function, and hence $g$ must be uniformly positive on $B_\delta \big(A_{\sigma}\setminus B_r(x)\big)\cap B_r(x)$ as well for $\delta>0$ small enough. 

In order to verify the final claim, we observe for $\sigma\neq \sigma_i(x)$
\begin{align*}
    H_i^0(x)\cap B_{r+\delta}(x)\cap A_\sigma=\emptyset
\end{align*}
by the property stated below \eqref{eq:dist}. Using \eqref{Eq:boost_by_epsilon} and the continuity of $g$, this argument can be extended to show that $H_i^0(x)\cap B_r(x)\cap B_\delta \big(A_{\sigma}\setminus B_r(x)\big)$ is empty as well.
\end{proof}

In order to avoid issues that arise due to our Banach space being possibly infinite dimensional, we will further use the subsequent auxiliary Proposition \ref{Prop:Reduction_to_finite}. We recall the enumeration in, and below, \eqref{Eq:enummeration_crit} of the critical points and of the corresponding elements $\{\eta_1,\dots ,\eta_m\}$.

\begin{proposition}
    \label{Prop:Reduction_to_finite}
Given $E,\rho>0$, we find a subspace $\mathcal{V}\subseteq B^*\cap A^{-1}(B^*)$ of finite dimension, containing $\{\eta_1,\dots ,\eta_m\}$, such that the corresponding orthogonal projection $\pi:B\longrightarrow H$ satisfies
    \begin{align*}
        F(x) & \geq F(\pi x)-\rho,
    \end{align*}
    as well as $\|x-\pi x\|_B  < \rho$, for all $x$ with $F(x)\leq E$.
\end{proposition}
\begin{proof}
   According to our assumptions we have that $\eta_j\in B^*\cap A^{-1}(B^*)$ and that $B^*\cap A^{-1}(B^*)$ is dense in $H$, and hence we find for any $y\in H$ and $\tau>0$ a subspace $\mathcal{V}\subseteq B^*\cap A^{-1}(B^*)$ containing $\{\eta_1,\dots ,\eta_m\}$, such that $\|(1-\pi_\mathcal{V})y\|_H<\tau$, where $\pi_\mathcal{V}$ is the corresponding projection. Given $\epsilon>0$, we find for any $y\in B$ such a subspace $\mathcal{V}(y)$ with $\|(1-\pi_{\mathcal{V}(y)}) \nabla U(y)\|_H< \epsilon$, and by continuity this property can be extended to a $B$-open neighborhood $ O_y$ containing $y$. We note that for $y\in [F\leq E]$, subspace $\mathcal{V}$ and $0\leq t\leq 1$, and $C$ large enough,
   \begin{align*}
      \| \pi_{\mathcal{V}} y+t(y-\pi_{\mathcal{V}} y)\|_H^2\leq \|y\|_H^2\leq 2F(y)+C\leq 2E+C.
   \end{align*}
   Taking a finite cover $O_{y_1},\dots ,O_{y_k}$ of the $B$-compact set $[F\leq 2E+C]$ and defining $\mathcal{V}_0:=\cup_{j=1}^k \mathcal{V}(y_j)$, we obtain for $\mathcal{V}$ with $\mathcal{V}_0\subseteq \mathcal{V}$ and any $y$ with $F(y)\leq E$
    \begin{align*}
        F(y) & =F(\pi_{\mathcal{V}} y)+\|y-\pi_{\mathcal{V}} y\|_H^2+\int_{0}^1  \big\langle y-\pi_{\mathcal{V}} y,\nabla U\big(\pi_{\mathcal{V}} y+t(y-\pi_{\mathcal{V}} y)\big)\big\rangle \mathrm{d}t\\
        & \geq F(\pi_{\mathcal{V}} y)+\|y-\pi_{\mathcal{V}} y\|_H^2-\epsilon \|y-\pi_{\mathcal{V}} y\|_H\\
        &\geq F(\pi_{\mathcal{V}} y)-\frac{\epsilon^2}{4}.
    \end{align*}
    Regarding the second property, we take again a finite cover of $[F\leq 2E+C]$ by $\bigcup_{j=1}^k B_{\rho/3}(y_j)$ and we denote with $b_{1},\dots b_k$ elements in $B^*\subseteq H$ with $\|b_j\|_{B^*}=1$ such that $\|y_j\|_B=b_j(y_j)$. By taking $\mathcal{V}$ possibly larger, we can assume 
    \begin{align*}
        \|(1-\pi_{\mathcal{V}})b_j\|_H\leq \frac{\rho}{3(2E+C)}.
    \end{align*}
    For any $y\in [F\leq E]$ we therefore have $y-\pi_\mathcal{V}y\in B_{\rho/3}(y_j)$ for some $j\in \{1,\dots ,m\}$, and we obtain together with $\|y\|_H^2\leq 2F(y)+C$ that
    \begin{align*}
        \|y_j\|_B & =|b_j(y_j)|\leq |b_j(y-\pi_\mathcal{V}y)|+\frac{\rho}{3}\leq \frac{\rho}{3(2E+C)} \|y-\pi_\mathcal{V}y\|_H+\frac{\rho}{3}\\
        & \leq \frac{\rho}{3(2E+C)} (F(y)+C)+\frac{\rho}{3}\leq \frac{2\rho}{3}.
    \end{align*}
    Consequently, $\|y-\pi_\mathcal{V}y\|_B< \|y_j\|_B+\frac{\rho}{3}\leq \rho$.
\end{proof}

We are now in a position to construct good approximate eigenstates. Let $r,\delta$ and $\epsilon$ be as in Proposition \ref{Cor:epsilon_half_space}, where we can assume by \eqref{eq:dist} w.l.o.g. that $r>0$ is small enough such that $x_\pm\in A_\pm\setminus B_r(\mathcal{T}_h)$ and $\delta>0$ is small enough such that the sets 
\begin{align*}
    \mathcal{A}_\pm := B_\delta\! \left(A_\pm\setminus B_r(\mathcal{T}_h)\right)
\end{align*}
 are disjoint (such a $\delta>0$ exists by \eqref{eq:dist}) and that $r>0$ is small enough such that also the sets $B_r(x_j)$ are disjoint for $j\in \{1,\dots,m\}$. We furthermore note that the open set
\begin{align*}
    \mathcal{O}_1:=\mathcal{A}_-\cup \mathcal{A}_+\cup B_r(\mathcal{T}_h)
\end{align*}
contains the compact set $[F\leq h]$, and hence there exists an open set $\mathcal{O}_2$ such that
\begin{align*}
    [F\leq h]\subseteq \mathcal{O}_2\subseteq \overline{\mathcal{O}_2}^B\subseteq \mathcal{O}_1.
\end{align*}
In particular, $\inf_{H\setminus \mathcal{O}_2}F=\min_{H\setminus \mathcal{O}_2}F>h+2\kappa$, for $\kappa$ small enough, and we choose a projection $\pi$ as in Proposition \ref{Prop:Reduction_to_finite} with $E:=h+\kappa$ and
\begin{align*}
    \rho:=\min\{\delta,r,\kappa\}.
\end{align*}
Since the embedding $\zeta:\mathrm{ran}\pi\longrightarrow B$ with $\zeta(y):=y$ is continuous, we have that $\zeta^{-1}(\mathcal{O}_1)$ and $\zeta^{-1}(\mathcal{O}_2)$ are open and
\begin{align*}
    \zeta^{-1}(\mathcal{O}_2)\subseteq \overline{\zeta^{-1}(\mathcal{O}_2)}^B\subseteq \zeta^{-1}(\mathcal{O}_1),
\end{align*}
and hence we find a $C^\infty$ function $\varphi$ on the finite dimensional subspace $\mathrm{ran}\pi$ such that $\varphi|_{\zeta^{-1}(\mathcal{O}_2)}=1$ and $\mathrm{supp}\varphi\subseteq \zeta^{-1}(\mathcal{O}_1)$. In particular
\begin{align*}
    \mathrm{supp}(\varphi\circ \pi) & \subseteq \pi^{-1}\mathcal{O}_1.
\end{align*}
Finally, let $\mathcal{T}_h=\{x_1,\dots ,x_m\}$ and let us introduce for $x_j\in \mathcal{T}_h$ the function $f_j:\mathbb R\longrightarrow [-1,1]$ as $f_j:=1$ in case $\sigma_1(x_j)=\sigma_2(x_j)=+$, $f_j:=-1$ in case $\sigma_1(x_j)=\sigma_2(x_j)=-$, and for $\sigma_1(x_j)=-$ and $\sigma_2(x_j)=+$ we define
\begin{align*}
    f_{\beta,j}(t):= \begin{cases}
       z(\epsilon,\beta)^{-1}\int_0^t e^{- \beta \alpha_j s^2/2-\frac{\chi(\epsilon^{-1} s)}{\epsilon^2 - s^2}} \mathrm{d}s, \text{ in case }|t|<\epsilon\\
       -1, \text{ in case }t\leq -\epsilon,\\
       1, \text{ in case }t\geq \epsilon,
    \end{cases} 
\end{align*}
where $z(\epsilon,\beta):=\int_0^\epsilon e^{- \beta \alpha_j s^2/2-\frac{\chi(\epsilon^{-1} s)}{\epsilon^2 - s^2}} \mathrm{d}s$, $\chi$ is a smooth function with $\chi(t)=0$ in a neighborhood of $t=0$ and $\chi(t)=1$ for $|t|>1/2$, and $-\alpha_j$ is the negative eigenvalue corresponding to $\eta_j$. Analogously, we define $f_j$ in case $\sigma_1(x_j)=+$ and $\sigma_2(x_j)=-$. We note that for a suitable constant $C$
\begin{align}
\label{eq:small_z_asy}
   \left|z(\epsilon,\beta)-\sqrt{\frac{\pi}{2\beta \alpha_j}}\right|\leq C/\beta.
\end{align}

In the subsequent Theorem \ref{theorem:same_hight} we show that the trial state
    \begin{align}
    \label{Eq:def_main_trial_state}
        \Psi_\beta(x):=\varphi(\pi x)\times \begin{cases}
            1, \text{ for }\pi x\in \mathcal{A}_+\\
            -1, \text{ for }\pi x\in \mathcal{A}_-\\
            f_{\beta,j}(\langle \eta_j,x-x_j \rangle), \text{ for }\pi x\in B_r(x_j)\text{ with }j\in \{1,\dots,m\}\\
            0, \text{ for }\pi x\notin \mathcal{O}_1
        \end{cases}
    \end{align}
    is a well-defined element in the domain of the operator $\mathcal{L}_\beta$ introduced in \eqref{Eq:Def_L_Operator} and we provide precise control on its energy 
\begin{align*}
    \mathcal{E}_\beta[\Psi_\beta]= \frac{1}{\beta}\|A^{1/2}\nabla \Psi_\beta\|_{L^2(\mu_\beta)}^2=\langle \Psi_\beta,\mathcal{L}_\beta\Psi_\beta\rangle,
\end{align*}
as well as on its second moment $\|\mathcal{L}_\beta \Psi_\beta\|_{L^2(\mu_\beta)}^2$.
To this end we shall use the following sharp Laplace asymptotics. 
\begin{proposition}
\label{Prop:Z_beta_asy}   Under Assumption~\ref{Ass_for_theorem_II}
for $\tau>0$ small enough
    \begin{align}
    \label{Eq:Control_local_masses}
        \int_{B_\tau(x_\pm)}e^{-\beta U}\mathrm{d}\gamma_\beta=e^{-\beta F(x_\pm)}C_\pm \left(1+O_{\beta\rightarrow \infty}\left(1/\sqrt{\beta}\right)\right).
    \end{align}
    Furthermore, we have for $\beta$ large enough and a suitable constant $C>0$
    \begin{itemize}
        \item in the case $F(x_-)=F(x_+)=0$, that 
         \begin{align}
         \label{Eq:estimate_C_-_C_+}
        \left|Z_\beta - \left(C_-+C_+\right)\right|\leq C/\sqrt{\beta},
    \end{align}
    \item in the case $0=F(x_-)<F(x_+)$
   \begin{align*}
       \left|Z_{\beta} - C_-\right|\leq C/\sqrt{\beta}.
   \end{align*}
    \end{itemize}
   In both cases $Z_\beta,Z_{\beta} \geq 1/2$ for $\beta$ large enough.
\end{proposition}
\begin{proof}
    Let us define the modified potential
    \begin{align*}
        U_\pm (x):=U(x+x_\pm ) - F(x_\pm) - \frac{1}{2}\langle x,\nabla^2   U(x_\pm)x \rangle  +  \langle x_\pm,x \rangle+\frac{1}{2}\|x_\pm\|^2_H,
    \end{align*}
    which satisfies $U_\pm(0)=0$, $DU_\pm(0)=0$ and $D^2 U_\pm (0)=0$, and hence we can find for all $q>0$ constants $\tau,C>0$ such that 
    \begin{align*}
        \left|e^{-\beta U_\pm (x)}-1\right|\leq \beta C\|x\|_B^3 e^{q\beta \|x\|_B^2}
    \end{align*}
for all $x\in B_\tau(0)$. We choose $q$ small enough such that $\|x\|_B^m e^{q\|x\|^2}\in L^1(\mathrm{d}\gamma_{\pm})$, and compute by applying the shift $x\mapsto x+x_\pm$
\begin{align*}
    \int_{B_\tau(x_\pm)}e^{-\beta U}\mathrm{d}\gamma_\beta=e^{-\beta F(x_\pm)}C_\pm \int_{B_\tau(0)}e^{-\beta U_\pm}\mathrm{d}\gamma_{\beta,\pm}.
\end{align*}
Consequently, we obtain for a suitable $C>0$
    \begin{align*}
      &  \ \ \ \ \  \left| \int_{B_\tau(x_\pm)}e^{-\beta U}\mathrm{d}\gamma_\beta-e^{-\beta F(x_\pm)}C_\pm \right|\\
      &  \leq e^{-\beta F(x_\pm)} C_\pm \int_{B_\tau(0)}|e^{-\beta U_\pm}-1|\mathrm{d}\gamma_{\beta,\pm}+e^{-\beta F(x_\pm)}\int_{B\setminus B_\tau(0)}\mathrm{d}\gamma_\beta\\
       &  \leq e^{-\beta F(x_\pm)}C_\pm C\beta^{-\frac{1}{2}}\int \|x\|_B^3 e^{q\|x\|_B^2}\mathrm{d}\gamma_{\pm}+e^{-\beta F(x_\pm)}e^{-q\beta \tau^2}\int  e^{q\|x\|_B^2}\mathrm{d}\gamma_{\pm}\\
       & \leq e^{-\beta F(x_\pm)}C /\sqrt{\beta},
    \end{align*}
    which concludes the proof of \eqref{Eq:Control_local_masses}. For \eqref{Eq:estimate_C_-_C_+} we note $\mathcal{C}_*:=B\setminus \left(B_\tau(x_-)\cup B_\tau(x_+)\right)$ is contained in $[F\geq u]$ for some $u>0$, and hence we have for $\beta$ large enough by the LDP
    \begin{align*}
        \int_{\mathcal{C}_*}e^{-\beta U}\mathrm{d}\gamma_\beta\leq e^{-\beta u/2},
    \end{align*}
    which concludes the proof together with \eqref{Eq:Control_local_masses}. The case $F(x_+)>0$, can be treated analogously.
\end{proof}

\begin{theorem}
\label{theorem:same_hight}
    Let us assume $U\in C^3(B)$ with $F(x_-)=F(x_+)=0$. The trial state $\Psi_\beta$ defined in \eqref{Eq:def_main_trial_state} is a well-defined element of $\operatorname{dom}(\mathcal{L}_\beta)$
   and there exists a constant $C>0$ such that for $\beta$ large enough, and $r,\kappa$ small enough, we have 
    \begin{align*}
         \|\mathcal{L}_\beta \Psi_\beta\|_{L^2(\mu_\beta)}^2  \leq  C e^{-\beta h}/\beta
    \end{align*}
 as well as
   \begin{align*}
         \frac{1}{\beta}\|A^{1/2}\nabla \Psi_\beta\|_{L^2(\mu_\beta)}^2= \left(1  +  O\! \left(\frac{1}{\sqrt{\beta}}\right)\right)\sum_{j=1}^{m_0} \frac{2C_j \alpha_j e^{-\beta h}}{\pi(C_-  +  C_+)} ,
   \end{align*}
Furthermore, there exists a $\tau>0$ small enough, such that $\Psi_\beta(x)=\pm 1$ for $x\in B_\tau(x_\pm)$.
\end{theorem}
\begin{proof}
    First of all we note that $\Psi_\beta:B\longrightarrow \mathbb R$ is well-defined as all elements $x\in B$ belong to at least one case appearing in the definition of $\Psi_\beta$, and $f_{\beta,j}(\langle \eta_j,x-x_j \rangle)=1$ in the intersection $\mathcal{A}_+\cap B_r(x_j)$, see Proposition \ref{Cor:epsilon_half_space} as well as the definition of $f_{\beta,j}$, and analogously $f_{\beta,j}(\langle \eta_j,x-x_j \rangle)=-1$ in the intersection $\mathcal{A}_-\cap B_r(x_j)$, and all other cases in the definition of $\Psi_\beta$ are disjoint to each other. Furthermore, $\Psi_\beta$ is in the domain of $\mathcal{L}_\beta$ as we can use an orthonormal basis $b_1,\dots,b_d\subseteq B^*\subseteq H$ of $\mathrm{ran}(\pi)$ in order to write
    \begin{align*}
        \Psi_\beta(x)=G_\beta(b_1(x),\dots,b_d(x))
    \end{align*}
    using $G_\beta \in C^\infty_b(\mathbb R^d)$ defined with $Y(y):=y_1 b_1+\dots +y_d b_d$ for $y\in \mathbb R^d$ as
    \begin{align*}
         G_\beta(y_1,\dots ,y_d):=\varphi(Y(y))\times \begin{cases}
            1, \text{ for }y\in Y^{-1}(\mathcal{A}_+)\\
            -1, \text{ for }y\in Y^{-1}(\mathcal{A}_-)\\
            f_{\beta,j}(\langle \eta_j,Y(y)-x_j \rangle), \text{ for }y\in Y^{-1}( B_r(x_j))\\
            0, \text{ for }y\notin \mathrm{supp}(\varphi\circ Y).
        \end{cases}
    \end{align*}
   Observe that $G_\beta$ is indeed a smooth function, as the different cases appearing in the definition of $G_\beta$ are all described by open sets. Using that $A|_{\mathrm{ran}(\pi)}:\mathrm{ran}(\pi)\longrightarrow H$ is bounded we therefore obtain
   \begin{align*}
       \|A^{1/2}\nabla \Psi_\beta(x)\|_H & \leq \|A|_{\mathrm{ran}(\pi)}\|_\mathrm{op} \|\nabla \Psi_\beta(x)\|_H\\
       &= \|A|_{\mathrm{ran}(\pi)}\|_\mathrm{op}\left(\sum_{k=1}^d (\partial_k G_\beta)^2(b_1(x),\dots b_d(x))\right)^{1/2} \leq C\beta^{1/2}, 
   \end{align*}
   for $C$ large enough, where we have used that $\beta^{-\frac{1}{2}}\partial_k G_\beta$ is uniformly bounded for $\beta\geq 1$, see \eqref{eq:small_z_asy}. Analogously we have by \eqref{Eq:Def_L_Operator}
\begin{align*}
    \mathcal{L}_\beta\Psi_\beta(x) & =-\frac{1}{\beta}\sum_{k,\ell=1}^d \langle b_k,A b_\ell\rangle \partial_k \partial_\ell G_\beta\big(b_1(x),\dots b_d(x)\big)\\
    &\ +   \sum_{\ell=1}^d
\bigl(\langle A b_\ell,x\rangle
+\langle A b_\ell,\nabla U(x)\rangle_H\bigr)
\partial_\ell G_\beta(b_1(x),\dots,b_d(x)), 
\end{align*}
and we obtain for $C,\tau>0$ large enough
   \begin{align}
              \label{Eq:L_Psi_bound}
       |\mathcal{L}_\beta\Psi_\beta(x)| &\leq C\beta^{1/2}e^{\tau \|x\|_B^2},
   \end{align}
  where we have used the assumption in \eqref{eq:expo_control}. Furthermore, 
    \begin{align}
    \label{eq:nabla_trial}
    \nabla \Psi_\beta & =\chi_{\mathcal{C}}\nabla \Psi_\beta+\sum_{j=1}^m \chi_{(\pi^{-1}B_r(x_j))\setminus \mathcal{C}}\nabla f_{\beta,j}(\langle \eta_j,x-x_j \rangle), \\
    \label{eq:L_trial}
        \mathcal{L}_{\beta}\Psi_\beta & =\chi_{\mathcal{C}}\mathcal{L}_{\beta}\Psi_\beta+\sum_{j=1}^m \chi_{(\pi^{-1}B_r(x_j))\setminus \mathcal{C}}\mathcal{L}_{\beta}f_{\beta,j}(\langle \eta_j,x-x_j \rangle),
    \end{align}
    with the $B$-closed set $\mathcal{C}:=(\pi^{-1}\mathcal{O}_2)^c\cup \{x:\|x-\pi x\|_B\geq \kappa\}$, where we have used that $\varphi\circ \pi=1$ on $\mathcal{C}^c\subseteq \pi^{-1}\mathcal{O}_2$. Notably, Proposition \ref{Prop:Reduction_to_finite} together with our choice $\rho\leq \kappa$, tells us that $\mathcal{C}\subseteq [F\geq h+\kappa]$ and we further observe
    \begin{align*}
        B_{r'}(x_j)\subseteq (\pi^{-1}B_r(x_j))\setminus \mathcal{C}\subseteq B_{r+\kappa}(x_j),
    \end{align*}
    for a suitable $r'$, where we used Proposition \ref{Prop:Reduction_to_finite} and the fact that $x_j\in [F\leq h+\kappa]$ to show $\|x_j-\pi x_j\|_B<\min\{\delta,r,\kappa\}$ and therefore $x_j$ is an element of the open set $(\pi^{-1}B_r(x_j))\setminus \mathcal{C}$. W.l.o.g. we choose $r'$ small enough, such that $t:=\langle \eta_j,x-x_j \rangle$ satisfies $|t|<\epsilon$ for all $x\in B_{r'}(x_j)$. We compute explicitly, using $\langle \eta_j,A\eta_j\rangle_H=1$,
    \begin{align*}
       A^{1/2} \nabla f_{\beta,j}(\langle \eta_j,x-x_j \rangle) & = f_{\beta,j}'(t)A^{1/2}\eta_j, \\
        \mathcal{L}_{\beta}f_{\beta,j}(\langle \eta_j,x-x_j \rangle) & =  \partial_{A\eta_j}F(x)f_{\beta,j}'(t)-\frac{1}{\beta}f_{\beta,j}''(t).
    \end{align*}
    We note that $|\left(\frac{\chi(\epsilon^{-1}t)}{\epsilon^2 - t^2}\right)'e^{-\frac{\chi(\epsilon^{-1}t)}{\epsilon^2 -t^2}}|\leq C$ for $C$ large enough and $|t|<\epsilon$, and
    \begin{align*}
        \left|\partial_{A\eta_j}F(x)+ \alpha_j t\right|\leq C\|x-x_j\|^2_B
    \end{align*}
    for all $x\in B_{r+\kappa}(x_j)$, given $r+\kappa$ is small enough, and $C$ large enough, where we used $U\in C^3(B)$ as well as $\nabla^2 F(x_j)A\eta_j=-\alpha_j \eta_j$, see the definition of $\eta_j:=\eta_{x_j}$ in \eqref{Eq:def_of_eta_x}, to show
    \begin{align*}
        \langle \nabla \partial_{A\eta_j}F(x_j),x-x_j\rangle_H= \langle \nabla^2 F(x_j)A\eta_j,x-x_j\rangle_H=-\alpha_j \langle \eta_j,x-x_j\rangle=-\alpha_j t.
    \end{align*}
    Together with \eqref{eq:small_z_asy}, an explicit computation shows for $|t|<\epsilon$
    \begin{align*}
        \left|\mathcal{L}_{\beta}f_{\beta,j}(\langle \eta_j,x-x_j \rangle)\right| & = z(\epsilon,\beta)^{-1}\left|\partial_{A\eta_j}F(x)+\alpha_j t+\frac{\chi(\epsilon^{-1}t)}{\beta (\epsilon^2 - t^2)}\right|e^{-\beta \alpha_jt^2/2-\frac{\chi(\epsilon^{-1}t)}{\epsilon^2 -t^2}}\\
        & \leq C\beta^{-1/2} \left(1+\beta \|x-x_j\|_B^2\right)e^{-\beta \alpha_j t^2/2},
    \end{align*}
    for $x\in B_{r+\kappa}(x_j)$. In combination with \eqref{eq:L_trial} and Proposition \ref{Prop:Z_beta_asy} we obtain
\begin{align}
\nonumber
   & \|\mathcal{L}_{\beta}\Psi_\beta\|_{L^2(\mu_\beta)}^2  \leq C\int_{\mathcal{C}}|\mathcal{L}_{\beta}\Psi_\beta|^2 e^{-\beta U}\mathrm{d}\gamma_\beta \\
   \label{eq:x_j_terms}
    & \ \ \ \ \ \ \ \ + C\beta^{-1} \sum_{j=1}^m \int_{B_{r+\kappa}(x_j)}\left(1+\beta \|x-x_j\|_B^2\right)^2e^{-\beta \left(U+ \alpha_j \langle \eta_j,x -x_j \rangle^2 \right)}\mathrm{d}\gamma_\beta(x),
\end{align}
    for a suitable $C>0$ and $\beta$ large enough. Since $\mathcal{C}$ is contained in $[F\geq h+\kappa]$, and $\|x\|_B^2\leq 2 F(x)+c$ for some $c>0$, we can pick $\theta>0$ small enough such that
        \begin{align*}
        \inf_{x\in \mathcal{C}}\left(F(x)-\theta \|x\|^2_B\right)>h+\kappa/2.
    \end{align*}
 Picking $\theta$ furthermore small enough such that $e^{\theta\|x\|_B^2}\in L^2(\mathrm{d}\gamma)$, we use that $\mathcal{C}$ is closed and obtain by the
 large deviations principle for $\gamma_\beta$ and the Varadhan lemma that for $\beta$ large enough
    \begin{align*}
        \int_{\mathcal{C}}|\mathcal{L}_{\beta}\Psi_\beta|^2 e^{-\beta U}\mathrm{d}\gamma_\beta\leq  C\beta \int_{\mathcal{C}} e^{\beta \theta \|x\|_B^2} e^{-\beta U}\mathrm{d}\gamma_\beta \leq e^{-\beta(h+\kappa/2)} .
    \end{align*}
Regarding the terms in \eqref{eq:x_j_terms} we first define the modified potential
\begin{align*}
   U_j(x):=U(x+x_j)-\frac{1}{2}\langle x,\nabla^2   U(x_j)x \rangle  +  \langle x_j,x \rangle+\frac{1}{2}\|x_j\|^2_H,
\end{align*}
which arises naturally once we shift the integral by the element $x_j$ in the Cameron-Martin space $H$ and employ the modified Gaussian measures $\gamma_{\beta,j}$
\begin{align*}
   & T_j:=\int_{B_{r+\kappa}(x_j)}\left(1+\beta \|x-x_j\|_B^2\right)^2e^{-\beta \left(U+\alpha_j \langle \eta_j,x -x_j \rangle^2 \right)}\mathrm{d}\gamma_\beta(x)\\
    & =\int_{B_{r+\kappa}(0)}\left(1+\beta \|x\|_B^2\right)^2e^{-\beta \left(U(x_j+x)+\alpha_j \langle \eta_j,x\rangle^2 \right)}e^{-\frac{\beta}{2}\|x_j\|^2_H-\beta \langle x_j,x \rangle}\mathrm{d}\gamma_\beta(x)\\
    & = C_j \int_{B_{r+\kappa}(0)}\left(1+\beta \|x\|_B^2\right)^2e^{-\beta U_j(x)}\mathrm{d}\gamma_{\beta,j}(x).
\end{align*}
Since $U_j\in C^3(B)$ with $DU_j(0)=0$ and $D^2 U_j(0)=0$, we have for some $C>0$
\begin{align*}
U_j(x)\geq U_j(0)-C\|x\|_B^3\geq h-C(r+\kappa)\|x\|_B^2  
\end{align*}
in case $x\in B_{r+\kappa}(x_j)$. Choosing $r+\kappa$ small enough such that $e^{C(r+\kappa)\|x\|_B^2}\in L^1(\mathrm{d}\gamma_j)$,
\begin{align*}
    T_j & \leq C_j e^{-\beta h}\int \left(1+\beta \|x\|_B^2\right)^2e^{\beta C(r+\kappa)\|x\|_B^2}\mathrm{d}\gamma_{\beta,j}(x)\\
    & =C_j e^{-\beta h}\int \left(1+\|x\|_B^2\right)^2e^{ C(r+\kappa)\|x\|_B^2}\mathrm{d}\gamma_{j}(x)\leq C e^{-\beta h},
\end{align*}
for a suitable $C$. Hence, we have for $\beta$ large enough and a suitable $C$
\begin{align*}
    \|\mathcal{L}_{\beta}\Psi_\beta\|_{L^2(\mu_\beta)}^2  \leq C e^{-\beta h}/\beta.
\end{align*}

We note at this point that $\Psi_\beta(x)=\pm 1$ for $x\in B_\tau(x_\pm)$ and $\tau$ small enough, since $\|x_\pm - \pi x_\pm\|<\delta$ by Proposition \ref{Prop:Reduction_to_finite} and hence $x_\pm$ is contained in the open set $\pi^{-1}\mathcal{A}_\pm$ and so is $B_\tau(x_\pm)\subseteq \pi^{-1}\mathcal{A}_\pm$ for $\tau>0$ small enough. 

Regarding the upper bound on $\|\nabla \Psi_\beta\|_{L^2(\mu_\beta)}^2$ we proceed similar to our analysis of $\mathcal{L}_{\beta}\Psi_\beta$ and estimate
\begin{align*}
  &  \|A^{1/2}\nabla \Psi_\beta\|_{L^2(\mu_\beta)}^2  \leq Z_\beta^{-1}\sum_{j=1}^m \int_{B_{r+\kappa}(x_j)}|f_{\beta,j}'(\langle \eta_j,x-x_j \rangle)|^2 e^{-\beta U}\mathrm{d}\gamma_\beta+ e^{-\beta(h+\kappa/2)}\\
    & \ \ \  \ \ \ \leq Z_\beta^{-1} z(\epsilon,\beta)^{-2}\sum_{j=1}^m C_j \int_{B_{r+\kappa}(0)} e^{-\beta U_j}\mathrm{d}\gamma_{\beta,j}+ e^{-\beta(h+\kappa/2)}\\
    &  \ \ \  \ \ \ \leq \left(1+C/\sqrt{\beta}\right)\sum_{j=1}^m \frac{2C_j \alpha_j \beta}{\pi(C_-+C_+)}\int_{B_{r+\kappa}(0)} e^{-\beta U_j}\mathrm{d}\gamma_{\beta,j}+ e^{-\beta(h+\kappa/2)}.
\end{align*}
Using $|e^{-\beta U_j}-e^{-\beta h}|\leq C \beta e^{-\beta h}\|x\|_B^3 e^{q \beta \|x\|^2}$ on $B_{r+\kappa}(0)$, we further have
\begin{align*}
   \int_{B_{r+\kappa}(0)} e^{-\beta U_j}\mathrm{d}\gamma_{\beta,j}\leq e^{-\beta h}+\beta^{-1/2}e^{-\beta h}\int e^{q\|x\|^2}\mathrm{d}\gamma_{j}.
\end{align*}
Finally, for the lower bound on $\|\nabla \Psi_\beta\|_{L^2(\mu_\beta)}^2$, we compute
\begin{align*}
    \|A^{1/2}\nabla \Psi_\beta\|_{L^2(\mu_\beta)}^2 & \geq Z_\beta^{-1}\sum_{j=1}^m \int_{B_{r'}(x_j)}|f_{\beta,j}'(\langle \eta_j,x-x_j \rangle)|^2 e^{-\beta U}\mathrm{d}\gamma_\beta\\
    &\geq Z\left(1 - C/\sqrt{\beta}\right)\sum_{j=1}^m \frac{2C_j \alpha_j \beta}{\pi(C_-+C_+)}\int_{B_{r'}(0)} e^{-\beta U_j}\mathrm{d}\gamma_{\beta,j}
\end{align*}
and use again $|e^{-\beta U_j}-e^{-\beta h}|\leq C \beta e^{-\beta h}\|x\|_B^3 e^{q \beta \|x\|^2}$ on $B_{r'}(0)$ to estimate
\begin{align*}
    \int_{B_{r'}(0)} e^{-\beta U_j}\mathrm{d}\gamma_{\beta,j} & \geq e^{-\beta h}-\int_{B\setminus B_{r'}(0)} \gamma_{\beta,j}-\beta^{-1/2}e^{-\beta h}\int e^{q\|x\|^2}\mathrm{d}\gamma_{j}\\
    & \geq e^{-\beta h}-\left(\beta^{-1/2}+e^{-\beta q (r')^2}\right)e^{-\beta h}\int e^{q\|x\|^2}\mathrm{d}\gamma_{j}.
\end{align*}
\end{proof}

With Theorem \ref{theorem:same_hight} at hand we can verify Theorem \ref{Th:precise_Asy}.
\begin{proof}[Proof of Theorem \ref{Th:precise_Asy}]
    In order to prove \eqref{Eq:precise_Asy_1}, let $\Psi^{(\pm)}_\beta:=(1\pm \Psi_\beta(x))/2$ and let us define the trial state
    \begin{align*}
        \Phi_\beta:=\sqrt{\frac{C_+}{C_-}}\Psi^{(-)}_\beta-  \sqrt{\frac{C_-}{C_+}}\Psi^{(+)}_\beta.
    \end{align*}
    We first note that $\mathcal{C}':=B\setminus \left(B_{\tau}(x_-)\cup B_{\tau}(x_+)\right)$ is contained in $[F\geq u]$ for some $u>0$, allowing us to estimate
\begin{align*}
  \int_{\mathcal{C}'}\mathrm{d}\mu_\beta  \leq e^{-\beta u/2},
\end{align*}
for $\beta$ large enough by the large deviations principle for $\gamma_\beta$ and the Varadhan lemma. Choosing $\tau>0$ as 
in Theorem \ref{theorem:same_hight}, i.e. such that $\Psi_\beta(x)=\pm 1$ for $x\in B_\tau(x_\pm)$, we obtain by Proposition \ref{Prop:Z_beta_asy}
    \begin{align}
            \label{Eq:Orth_cond}
       |\langle 1,\Phi_\beta\rangle_{L^2(\mu_\beta)}|\leq & C/\sqrt{\beta},\\
       \label{Eq:Norm_cond} 
        \|\Phi_\beta\|_{L^2(\mu_\beta)}= & 1+O_{\beta\rightarrow \infty}\! \left(1/\sqrt{\beta}\right).
    \end{align}
    Furthermore, using the properties of $\Psi_\beta$ derived in Theorem \ref{theorem:same_hight}, we have
    \begin{align}
              \label{Eq:Var_cond}
         \|\mathcal{L}_{\beta} \Phi_\beta\|_{L^2(\mu_\beta)}^2  \leq & Ce^{-\beta h}/\beta\\
\nonumber 
          \left\langle \Phi_\beta,\mathcal{L}_\beta \Phi_\beta\right\rangle & =\frac{1}{\beta}\|A^{1/2}\nabla \Phi_\beta\|_{L^2(\mu_\beta)}^2   \\
                    \label{Eq:Exp_cond}
        &  =    \left(1 + O_{\beta\rightarrow \infty}\! \left(1/\sqrt{\beta}\right)\right)\sum_{j=1}^{m_0} \frac{(C_- + C_+ )C_j \alpha_j e^{-\beta h}}{2\pi C_- C_+} .
    \end{align}
    Using that $\lambda^{(1)}_\beta $ is the only eigenvalue of $\mathcal{L}_\beta$ in $(0,\delta)$ for a $\delta>0$ small enough, see Theorem \ref{Thm:Main_Abstract_Theorem}, we obtain by the Kato-Temple inequality for the trial state 
    \begin{align*}
        \omega_\beta:=\frac{\Phi_\beta}{\|\Phi_\beta\|_{L^2(\mu_\beta)}}
    \end{align*}
the following lower bound 
    \begin{align}
    \label{Eq:Kato_Temple}
   \lambda^{(1)}_\beta\geq  \left\langle \omega_\beta,\mathcal{L}_\beta \omega_\beta\right\rangle-\frac{\left\|\mathcal{L}_\beta \omega_\beta\right\|^2-\left\langle \omega_\beta,\mathcal{L}_\beta \omega_\beta\right\rangle^2}{\delta-\left\langle \omega_\beta,\mathcal{L}_\beta \omega_\beta\right\rangle}  \geq \left\langle \omega_\beta,\mathcal{L}_\beta \omega_\beta\right\rangle-\frac{\left\|\mathcal{L}_\beta \omega_\beta\right\|^2}{\delta-\left\langle \omega_\beta,\mathcal{L}_\beta \omega_\beta\right\rangle} .
    \end{align}
   By \eqref{Eq:Norm_cond} and \eqref{Eq:Exp_cond}, we know that
    \begin{align*}
        \left\langle \omega_\beta,\mathcal{L}_\beta \omega_\beta\right\rangle=\frac{ \left\langle \Phi_\beta,\mathcal{L}_\beta \Phi_\beta\right\rangle }{ \|\Phi_\beta\|_{L^2(\mu_\beta)}^2}=\left(1 + O_{\beta\rightarrow \infty}\! \left(1/\sqrt{\beta}\right)\right)\sum_{j=1}^{m_0} \frac{(C_- + C_+ )C_j \alpha_j e^{-\beta h}}{2\pi C_- C_+} .
    \end{align*}
    In combination with \eqref{Eq:Var_cond} we therefore obtain by \eqref{Eq:Kato_Temple}
    \begin{align*}
        \lambda^{(1)}_\beta\geq \left(1 + O_{\beta\rightarrow \infty}\! \left(1/\sqrt{\beta}\right)\right)\sum_{j=1}^{m_0} \frac{(C_- + C_+ )C_j \alpha_j e^{-\beta h}}{2\pi C_- C_+} .
    \end{align*}
    Regarding the upper bound, we define the state 
    \begin{align*}
        \omega'_\beta:=\frac{\Phi_\beta-\langle 1,\Phi_\beta\rangle_{L^2(\mu_\beta)}}{\sqrt{\|\Phi_\beta\|_{L^2(\mu_\beta)}^2-|\langle 1,\Phi_\beta\rangle_{L^2(\mu_\beta)}|^2}},
    \end{align*}
and note that $\lambda^{(1)}_\beta\leq \langle \omega'_\beta,\mathcal{L}_\beta.\omega'_\beta\rangle$, since $\omega'_\beta$ is normed and orthogonal to $1$, which is the eigenstate to the lowest eigenvalue $0$ of $\mathcal{L}_\beta $, i.e. $\mathcal{L}_\beta 1=0$. By \eqref{Eq:Orth_cond} and \eqref{Eq:Exp_cond}
\begin{align*}
     \langle \omega'_\beta,\mathcal{L}_\beta \omega'_\beta\rangle & =\frac{\langle \Phi_\beta,\mathcal{L}_\beta \Phi_\beta\rangle}{\|\Phi_\beta\|^2_{L^2(\mu_\beta)}-|\langle 1,\Phi_\beta\rangle_{L^2(\mu_\beta)}|^2}\\
     & =\left(1 + O_{\beta\rightarrow \infty}\! \left(1/\sqrt{\beta}\right)\right)\sum_{j=1}^{m_0} \frac{(C_- + C_+ )C_j \alpha_j e^{-\beta h}}{2\pi C_- C_+} ,
\end{align*}
which concludes the proof of \eqref{Eq:precise_Asy_1}. The other statement \eqref{Eq:precise_Asy_2} follows analogously, using the trial state $\Phi_\beta:=\Psi_\beta^{(+)}$ instead.
\end{proof}

\section{Examples}     \label{sec5}

\subsection{Allen-Cahn}

We consider on the Banach space $B:=C([0,\ell])$ the Gaussian measure $\gamma$ with corresponding Cameron-Martin space $H:=H^1((0,\ell))\subseteq B$, i.e. according to the natural embedding $B^*\subseteq H^*=H$ the measure $\gamma$ is characterized by
\begin{align*}
     \int_B  e^{ib(x)} d\gamma    =    e^{-\|b\|_{H^1((0,\ell))}^2/2}.
\end{align*}
On the Hilbert space $H^1((0,\ell))$ we then study the functional 
\begin{align}
\label{Eq:F_Introduction}
    F(x):=\frac{1}{2}\int_0^\ell x'(s)^2\mathrm{d}s+\int_0^\ell u(x(s))\mathrm{d}s,
\end{align}
with the function $u:\mathbb R\longrightarrow \mathbb R$ defined as
\begin{align*}
    u(z):=\frac{1}{4}z^4-\frac{1}{2}z^2+\frac{1}{4}.
\end{align*}
Notably, we can write $F(x)=\frac{1}{2}\|x\|_{H^1((0,\ell))}^2+U(x)$ with
\begin{align*}
    U(x):=\int_0^\ell u_*(x(s))\mathrm{d}s:=\int_0^\ell \left(u(x(s))-\frac{1}{2}x(s)^2\right)\mathrm{d}s.
\end{align*}
Let us define the operator $K$ acting on $L^2((0,\ell))$ as $K:=-\left(\frac{\mathrm{d}}{\mathrm{d}t}\right)^2+1$ with free boundary conditions, and let $A$ be the self-adjoint operator acting on $H^1((0,\ell))$ given by the restriction of $K$ to $H^3((0,\ell))\cap \mathrm{dom}(K)$.     We then consider the quadratic form corresponding to the stochastic Allen-Cahn equation
\begin{align*}
  \mathcal E^{(\mathrm{AC})}_{\beta}[f] : =    \frac{1}{\beta} \int_B \|\ \sqrt A \nabla f\|_{H^1((0,\ell))}^2  d\mu^{(\mathrm{AC})}_\beta ,
\end{align*}
where the measure $\mu^{(\mathrm{AC})}_\beta$ is defined according to \eqref{Eq:def_mu}. The critical points $x\in \mathcal{S}$ are given by the solutions
\begin{align*}
    0=\nabla F(x)=x+K^{-1}(x^3-2x),
\end{align*}
i.e. $x$ is in the domain of $K$ and we have $0=K\nabla F(x)=-x''+x^3-x$, or equivalently $x$ satisfies the stationary Allen-Cahn equation
\begin{align}
\label{Eq:stationary_AC}
    x''=x^3-x.
\end{align}
The solutions to \eqref{Eq:stationary_AC} have been analyzed in \cite{BerglundGentz2013} depending on the length $\ell$. For $n\pi<\ell\leq (n+1)\pi$, with $n\in \mathbb N$, there are exactly three constant solutions $x_-:=-1$, $x_0:=0$ and $x_+:=1$, as well as $2n$ non-constant solutions $\{x_1,\dots ,x_{2n}\}$. Introducing according to \eqref{Eq:Q_Operator} the operator $Q_x:=-\left(\frac{\mathrm{d}}{\mathrm{d}t}\right)^2+3x^2-1$ acting on $L^2((0,\ell))$, we note 
\begin{align*}
    Q_{x_\pm}=-\left(\frac{\mathrm{d}}{\mathrm{d}t}\right)^2+2
\end{align*}
is strictly positive, and $Q_{x_0}=-\left(\frac{\mathrm{d}}{\mathrm{d}t}\right)^2-1$ has exactly $n+1$ negative eigenvalues for $n\pi<\ell\leq (n+1)\pi$ and a trivial kernel as long as $\ell \notin \pi \mathbb Z$. Observe that the lowest eigenvalue equals $-1$. Furthermore, according to \cite{BerglundGentz2013}, the non-constants solutions $\{x_1,\dots ,x_{2n}\}$ can be arranged such that $x_{2k-1}=-x_{2k}$, and the operators $Q_{x_{2k-1}}=Q_{x_{2k}}$ have exactly $k$-many negative eigenvalues and a trivial kernel. Since the operator $Q_x$ is a Schrödinger Operator defined on a compact domain, it is clear that $Q_x$ is invertible as long as the kernel is trivial. In this case $\nabla^2F(x)$ is invertible too with inverse given by
\begin{align}
\label{Eq:Inverse_nabla_square_F}
    Q^{-1}A=(-\Delta +3x^2-1)^{-1}(-\Delta+1)=1-Q^{-1}(3x^2-2).
\end{align}
Indeed, $Q^{-1}(3x^2-2)$ is a bounded operator, as the multiplication by the smooth function $3x^2-2$ is a bounded operator on $H^1((0,\ell))$ and so is $Q^{-1}$. Viewing the dual space $B^*$ as a subset of $H^1((0,\ell))$, we can further embed the space $L^2((0,\ell))$ in $B^*$ via the map $\iota(y):=K^{-1}(y)$, i.e. $\mathrm{dom}(K)\subseteq B^*$. Since $\mathrm{dom}(A)\subseteq \mathrm{dom}(K)$,
\begin{align*}
A^{-1}(B^*)\cap B^*=A^{-1}(B^*)\supseteq A^{-1}(\mathrm{dom}(A))=\mathrm{dom}(A^{2})
\end{align*}
is dense in $H^1((0,\ell))$ and $\eta_x:=A^{-1}\nu_x$ defined in \eqref{Eq:def_of_eta_x} is an element of $\mathrm{dom}(A)\subseteq  B^*$.

Denoting with $\lambda^{(\mathrm{AC})}_{\beta,j}$ the minmax values of $ \mathcal E^{(\mathrm{AC})}_{\beta}$, we obtain as a consequence of Theorem \ref{Thm:Main_Abstract_Theorem} and Theorem \ref{Th:precise_Asy} the following corollary.
\begin{corollary}
\label{Cor:AC_result}
Let $\ell>0$. Then
\[
   \liminf_{\beta\to\infty}\lambda_{\beta,2}^{(\mathrm{AC})}>0 .
\]
Furthermore, as $\beta\to\infty$, if $0<\ell<\pi$, then
\begin{equation} 
   \label{Eq:AC_simple}
   \lambda_{\beta,1}^{(\mathrm{AC})}
   =
   \frac1\pi
   \sqrt{
      \frac{\det[\nabla^2F(x_-)]}
           {|\det[\nabla^2F(x_0)]|}
   }\,
   e^{-\beta\ell/4}
   \left(1+O(\beta^{-1/2})\right).
\end{equation}
If $n\pi<\ell<(n+1)\pi$ for some $n\in\mathbb N$, $n\ge1$, then
\[
   \lambda_{\beta,1}^{(\mathrm{AC})}
   =
   \frac{2\alpha}{\pi}
   \sqrt{
      \frac{\det[\nabla^2F(x_-)]}
           {|\det[\nabla^2F(x_1)]|}
   }\,
   e^{-\beta F(x_1)}
   \left(1+O(\beta^{-1/2})\right),
\]
where $-\alpha<0$ is the unique negative eigenvalue of $Q_{x_1}$.
\end{corollary}
\begin{remark}
    (a) Since in this special case $U\in C^4(B)$ the error term $O  ( \beta^{-1/2} )   $   can be improved 
    to   $ O  ( \beta^{-1} )$.  \\
    (b) In case $0<\ell<\pi$ the formula in \eqref{Eq:AC_simple} can be simplified further, as the eigenvalues of $\nabla^2 F(x_-)$ are explicitly given by $\frac{\pi^2k^2+2\ell^2}{\pi^2 k^2+\ell^2}$ for $k \in \mathbb Z$, and similarly the eigenvalues of $\nabla^2 F(x_0)$ are given by $\frac{\pi^2k^2-\ell^2}{\pi^2 k^2+\ell^2}$. Explicitly, we obtain
    \begin{align*}
        \lambda^{(\mathrm{AC})}_{\beta,1} & =     \frac{1}{\pi}\sqrt{\prod_{k\in \mathbb Z}\frac{\pi^2k^2+2\ell^2}{|\pi^2k^2-\ell^2|}}e^{-\frac{\beta \ell}{4}}           \left(1+O  ( \beta^{-1/2} )\right)        \\
        & =      \frac{\sqrt{2}\sinh(\sqrt{2}\ell)}{\pi\sin(\ell)}e^{- \beta \ell /4}          \left(1+O  ( \beta^{-1/2} )\right)     ,
    \end{align*}
  where we have used Euler's product formula to express the infinite product
    \begin{align*}
       \sqrt{\prod_{k\in \mathbb Z}\frac{\pi^2k^2+2\ell^2}{|\pi^2k^2-\ell^2|}}=\sqrt{2}\prod_{k=1}^\infty \frac{\pi^2k^2+2\ell^2}{\pi^2k^2-\ell^2}=\frac{\sqrt{2}\sinh(\sqrt{2}\ell)}{\sin(\ell)}.
    \end{align*} \\
    (c) In the case of periodic boundary conditions, instead of the free boundary conditions studied in Corollary \ref{Cor:AC_result}, the analysis in \cite{BerglundGentz2013} yields that again $x_-$ and $x_+$ are the only critical points with $\nabla^2 F(x_\pm)\geq 0$, and that $\nabla^2 F(x_\pm)$ is invertible. Hence, it is easy to see that for periodic boundary conditions the Assumption \ref{Ass_for_theorem_I} is satisfied, and hence Theorem \ref{Thm:Main_Abstract_Theorem} is applicable in this case as well, yielding 
    \begin{align*}
        \varliminf_{\beta\rightarrow \infty}\lambda^{(\mathrm{AC},\mathrm{periodic})}_{\beta,2}    >0.
    \end{align*}
    However, Theorem \ref{Th:precise_Asy} is not directly applicable for $\ell>\pi$, as in this case there is a translation-invariant orbit of critical points which have exactly one negative eigenvalue and a non-trivial kernel. 
    \begin{align*}
        \{x_{t}:t\in [0,\ell)\}.
    \end{align*} 
  \\
(d) Similar as in \cite{BerglundGentz2013}, our result is not specific to the concrete form of the potential $u(z):=\frac{1}{4}z^4-\frac{1}{2}z^2+\frac{1}{4}$. To be precise, we might consider any $u\in C^2(\mathbb R)$ which satisfies: there are exactly two local minima $x_\pm$ and one maximum $x_0$, such that $u''$ is non-zero at all three stationary points ( w.l.o.g. let $x_-<x_0=0<x_+$, $u(0)=0$ and $u''(0)=-1$), and we assume there exists a $p_0\in \mathbb N_{\geq 2}$ 
\begin{itemize}
    \item such that 
    \begin{align*}
        |u^{(2)}(x)|\leq C (1+|x|^{2(p_0-1)}),
    \end{align*}
    \item and such that for some $a,c>0$
    \begin{align*}
        \mathrm{sgn}(x)u'(x)\geq a |x|^{2p_0-1}-c,
    \end{align*}
    \end{itemize}
and that the period $T(E)$ as a function of $E$ satisfies
    \begin{align*}
        T'(E)>0.
    \end{align*}
     It has been verified in \cite{BerglundGentz2013} that the assumptions above are always satisfied, in case we have for all $x\in (x_-,x_+)\setminus \{0\}$
    \begin{align*}
        (u')^2>2u u''.
    \end{align*}
  In particular, the assumptions are met for all potentials $u(x):=\lambda x^{2m}-\frac{x^2}{2}$ with $\lambda>0$ and $m\in \mathbb N_{\geq 2}$.
\end{remark}
\begin{proof}[Proof of Corollary \ref{Cor:AC_result}]
    In the following we have to verify that Assumptions \ref{Ass_for_theorem_I} and \ref{Ass_for_theorem_II} hold. We first note that the additional Assumptions stated in Theorem \ref{Th:precise_Asy} hold, as $\eta_x\in B^*$ and $U\in C^3(B)$, the second property follows from the multi-linear form
    \begin{align*}
        (v_1,v_2,v_3)\mapsto \partial_{v_1}\partial_{v_2}\partial_{v_3}U(x)=\int_0^\ell u_*'''(x(s))v_1(s)v_2(s)v_3(s)\mathrm{d}s
    \end{align*}
    being bounded and depending continuously on $x\in B$. We verify explicitly
    \begin{align}
    \nonumber
     \sup_{v:\|v\|_{\infty}=1} |\partial_{v}^3U(x)-\partial_{v}^3U(y)|  &  = \sup_{\eta:\|v\|_{\infty}=1}\int_{0}^\ell |v(s)|^3 \big|u_*'''(y(s))-u_*'''(x(s))\big|\mathrm{d}s\\
     \nonumber
     &  \leq \int_{0}^\ell \big|u_*'''(y(s))-u_*'''(x(s))\big|\mathrm{d}s\\
     \label{Eq:Final_Step_Second_Der_Cont}
     & \leq \ell \sup_{|x|\leq \|x\|_\infty,|z|\leq \|y-x\|_\infty}|u_*'''(x+z)-u_*'''(x)|,
    \end{align}
     with the right hand side in Eq.~(\ref{Eq:Final_Step_Second_Der_Cont}) converging to zero as $\|y-x\|_\infty\longrightarrow 0$. Regarding Assumption \ref{Ass_for_theorem_I}, it is clear that $U$ is lower bounded as $u_*$ is as well and based on the discussion above \eqref{Eq:Inverse_nabla_square_F} we know that $\nabla^2 F(x_\pm)\geq c>0$, and similar to the computation in \eqref{Eq:Final_Step_Second_Der_Cont} we have $\|DU(x)\|\leq \ell \|u_*'(x(s))\|_\infty\leq C(1+\|x\|_\infty)^3$. Denoting with $\mathcal{M}$ the multiplication operator by $x\mapsto u''(x(s))$, we further obtain
   \begin{align*}
       \beta^{-1}V_\beta(x) & =\frac{1}{4}\|\nabla U\|_H^2+\frac{1}{2}\langle DU(x),x\rangle-\frac{1}{2\beta}\mathrm{tr}_H\nabla^2 U\\
        & \geq \frac{1}{2}\langle DU(x),x\rangle-\frac{1}{2\beta}\mathrm{tr}_H\nabla^2 U\\
        &= \int_{0}^\ell x(s)u_*'(x(s))\mathrm{d}s - \frac{1}{2\beta}\mathrm{tr}_{L^2}\! \left[K^{ -1}\mathcal{M}\right]\\
                &= \int_{0}^\ell x(s)u_*'(x(s))\mathrm{d}s - \frac{1}{2\beta}\mathrm{tr}_{L^2}\! \left[K^{ -1}\right]\frac{1}{\ell}\int_{0}^\ell u_*''(x(s))\mathrm{d}s\\
        &\geq \int_{0}^\ell \left\{x(s)^{2p_0}+c- \frac{C}{\ell \beta}\mathrm{tr}_{L^2}\! \left[K^{-1}\right]x(s)^{2p_0-2}\right\}\mathrm{d}s\\
        & \geq \ell \left(c-\left(\frac{C}{2\ell \beta}\mathrm{tr}_{L^2}\! \left[K^{ -1}\right]\right)^{p_0}\right),
   \end{align*}
and in particular \eqref{Eq:Th_Assumption} holds. Regarding Assumption \ref{Ass_for_theorem_II}, we only need to verify \eqref{eq:assump_comp} as the other assumptions are discussed around \eqref{Eq:Inverse_nabla_square_F}. We compute for critical points $x$
\begin{align}
\nonumber
& \langle y ,  \nabla^2 F(x)y\rangle_{H^1((0,\ell))}  =\|y\|^2_{H^1((0,\ell))}+ \langle y , A^{-1}(3x^2-1)y \rangle_{H^1((0,\ell))}\\
\label{Eq:comp_bound_A_square}
 & \ \ \ \ \ \geq \|y\|^2_{H^1((0,\ell))}- \epsilon \langle y , (3x^2-1)^2 y \rangle_{H^1((0,\ell))}- \epsilon^{-1}\langle y , A^{-2}y \rangle_{H^1((0,\ell))}\\
 \nonumber
&  \ \ \ \ \ \geq \|y\|^2_{H^1((0,\ell))}- \epsilon \langle y , (3x^2-1)^2 y \rangle_{H^1((0,\ell))}- a^{-1}\epsilon^{-1}\langle y , A^{-1}y \rangle_{H^1((0,\ell))},
\end{align}
which concludes the proof for $\epsilon>0$ small enough, as the multiplication by $(3x^2-1)^2$ is a bounded operator on $H^1((0,\ell))$.
\end{proof}

\subsection{Cahn-Hilliard}

Here the Banach space $\widetilde{B}\subseteq C([0, \ell])$ consists of functions with zero average and $\widetilde{\gamma}$ is the centered Gaussian measure with corresponding Cameron-Martin space 
\begin{align*}
   \dot{H}^1((0,\ell)):=\left\{y\in H^1((0,\ell)):\int y(s)\mathrm{d}s=0\right\},
\end{align*}
equipped with the norm $\|y\|_{\dot{H}^1((0,\ell))}:=\int (y'(s))^2\mathrm{d}s$. Furthermore, let us introduce $\widetilde{K}x:=-x''$ defined on the Hilbert space $\{y\in L^2((0,\ell)):\int y(s)\mathrm{d}s=0\}$ with free boundary conditions, and denote with $\widetilde{A}$ the restriction of the operator $(\widetilde{K})^2$ to the domain $H^5((0,\ell))\cap \mathrm{dom}(\widetilde{K}^2)$. We then consider the Allen-Cahn functional in \eqref{Eq:F_Introduction} restricted to $\widetilde{B}$
\begin{align*}
    \widetilde{F}(x)=\frac{1}{2}\|x\|^2_{\dot{H}^1((0,\ell))}+\int_0^\ell \left(\frac{1}{4}x(s)^4-\frac{1}{2}x(s)^2+\frac{1}{4}\right)\mathrm{d}s,
\end{align*}
and denote the corresponding measure and quadratic form by $\mu_\beta^{(\mathrm{CH})}$ and
\begin{align*}
    \mathcal E^{(\mathrm{CH})}_{\beta}[f] : =    \frac{1}{\beta} \int_{\widetilde{B}} \|\ \sqrt{\widetilde{A}} \nabla f\|_{\dot{H}^1((0,\ell))}^2  d\mu^{(\mathrm{CH})}_\beta.
\end{align*}
Furthermore, let $\lambda^{(\mathrm{CH})}_{\beta,j}$ be the minmax values of $ \mathcal E^{(\mathrm{CH})}_{\beta}$. Before we state the asymptotic behavior of $\lambda^{(\mathrm{CH})}_{\beta,j}$, we analyze the critical points of $\widetilde{F}$ in the subsequent Proposition \ref{Prop:Critical_points_CH}.
\begin{proposition}
\label{Prop:Critical_points_CH}
    Let $\ell=\pi+\epsilon$.  Then, for $\epsilon>0$ small enough, the critical points of $\widetilde{F}$ are given by the solutions $x$ of the stationary Allen-Cahn equation satisfying $\int x(t)\mathrm{d}t=0$, which for $0<\epsilon<\pi$ are given by $\{x_0,x_1,x_2\}$, where $x_0:=0$ and $x_2(t)=x_1(\ell-t)=-x_1(t)$. Furthermore, $x_1$ and $x_2$ are the global minima, satisfying $\widetilde{F}(x_1)=\widetilde{F}(x_2)$ and $\nabla^2\widetilde{F}(x_1),\nabla^2\widetilde{F}(x_2)\geq c>0$, and the Hessian $\nabla^2\widetilde{F}(x_0)$ is invertible and has exactly a single negative eigenvalue. 
\end{proposition}
\begin{proof}
    It is clear that any stationary solution of the Allen-Cahn equation with $\int x(t)\mathrm{d}t=0$ is a critical point of $\widetilde{F}$. On the other hand, if $x\in \dot{H}^1((0,\ell))$ is a critical point of $\widetilde{F}$, then $x$ satisfies Neumann boundary conditions and there exists a $\lambda\in \mathbb R$ with
    \begin{align}
    \label{Eq:EL_eq_orthogonal}
        x''=x^3-x+\lambda.
    \end{align}
    In the following we want to show that $\lambda=0$, i.e. that $x$ is a stationary solution of the Allen-Cahn equation. In a first step we are going to derive the apriori bounds $\|x\|_\infty+|\lambda|\leq C\epsilon^{1/4}$ for a suitable $C>0$. Using $\int x(t)\mathrm{d}t=0$, we note
\begin{align*}
    \int (x')^2\mathrm{d}t\geq \frac{\pi}{\pi+\epsilon}\int x(t)^2\mathrm{d}t.
\end{align*}
Together with \eqref{Eq:EL_eq_orthogonal} this yields
\begin{align*}
    0\geq \int x(t)^4\mathrm{d}t-\frac{\epsilon}{\pi+\epsilon}\int x(t)^2\mathrm{d}t,
\end{align*}
and in particular $\int x(t)^2\mathrm{d}t\leq C\epsilon$ for a suitable constant $C>0$. Similarly, we have $\int (x')^2\mathrm{d}t\leq C$, and interpolation yields the rough, but sufficient, estimate
\begin{align*}
  |x(0)|\leq \|x\|_\infty\leq C \epsilon^{1/4}  ,
\end{align*}
for some constant $C>0$. We further note that, since $x'(0)=x'(\pi+\epsilon)=0$, there exists a $t\in (0,\pi+\epsilon)$ such that $x''(t)=0$, and therefore
\begin{align*}
    0=x(t)^3-x(t)+\lambda.
\end{align*}
As a consequence, we obtain $|\lambda|\leq \|x\|_\infty\leq C\epsilon^{1/4}$ as well.

We furthermore note that in case $r:=x(0)=0$, then, depending on the sign of $\lambda$, the solution $x(t)$ lives in either $[0,\infty)$ or $(-\infty,0]$. However, as $\int x(t)\mathrm{d}t=0$, this is only possible given that $x=0$, which is a solution of the stationary Allen-Cahn equation. Hence, we can assume in the following that $r>0$. Similarly, we note that for $0<r<r_0$ and $|r-\lambda|< \tau r$, the solution satisfies
\begin{align*}
    x(t)=(1+o_{\tau,r\rightarrow 0}(1))r\geq r/2>0,
\end{align*}
for $\tau$ and $r$ small enough, which is a contradiction to $\int x(t)\mathrm{d}t=0$. Consequently, we can assume in the following
\begin{align}
   \label{Eq:ass_lambda_r}
    |r-\lambda| \geq \tau r>0.
\end{align}

In the following let $x_{\lambda,r}$ denote a general solution of \eqref{Eq:EL_eq_orthogonal} with $x_{\lambda,r}(0)=r$ and $x_{\lambda,r}'(0)=0$, i.e. $x=x_{\lambda,x(0)}$. By a perturbative argument it is clear that
\begin{align}
\label{Eq:approx_in_lambda}
    x^{(n)}_{\lambda,r}=x^{(n)}_{0,r}+\lambda \varphi^{(n)}_{r}+o_{\lambda\rightarrow 0}(\lambda),
\end{align}
where $\varphi_{r}$ satisfies $\varphi_{r}(0)=\varphi_{r}'(0)=0$ and solves
\begin{align*}
    \varphi_{r}''=3x_{0,r}^2\varphi_{r}-\varphi_{r}+1.
\end{align*}
Similarly, we have that
\begin{align}
\label{Eq:approx_in_r}
    \varphi_r & =\varphi_0+o_{r\rightarrow 0}(1)=1-\cos(t)+o_{r\rightarrow 0}(1),\\
    \label{Eq_compare_harmonic}
    x^{(n)}_{0,r}(t) & =r\cos^{(n)}(t)+r^3\left(\int_0^t \sin(t-s)\cos(s)^3\mathrm{d}s\right)^{(n)}+o_{r\rightarrow 0}(r^3).
\end{align}
Let $T_{\lambda,r}:=\inf\{t>0:x'_{\lambda,r}(t)=0\}$. By \eqref{Eq_compare_harmonic}, and the mean value theorem,
\begin{align}
\label{Eq:for_T_0_r}
   T_{0,r}=\pi+\frac{3\pi r^2}{8}+o_{r\rightarrow 0}(r^2). 
\end{align}
Together with \eqref{Eq:approx_in_lambda} and \eqref{Eq:approx_in_r} for the first derivative $n=1$, we obtain 
\begin{align*}
     x'_{\lambda,r}(T_{0,r}\pm \kappa)=\pm (\lambda-r+o_{r,\kappa\rightarrow 0}(r))\kappa+o_{\lambda,r\rightarrow 0}(\lambda).
\end{align*}
By assumption \eqref{Eq:ass_lambda_r} we have $|r-\lambda|\geq \tau \max\{r,|\lambda|\}$, and hence we find $\kappa$ of the order $o_{\lambda,r\rightarrow 0}\! \left(\min\{1,\frac{\lambda}{r}\}\right)$ such that $\pm x'_{\lambda,r}(T_{0,r}\pm \kappa)\geq 0$. By the mean value theorem, we therefore obtain that
\begin{align*}
    T_{\lambda,r}=T_{0,r}+o_{\lambda,r\rightarrow 0}\! \left(\frac{\lambda}{r}\right).
\end{align*}
Using $x_{0,r}(T_{0,r}-t)=-x_{0,r}(t)$ therefore yields
\begin{align*}
    \int_0^{T_{\lambda,r}} x_{0,r}\mathrm{d}t=\int_{T_{0,r}}^{T_{\lambda,r}} x_{0,r}\mathrm{d}t=o_{\lambda,r\rightarrow 0}(\lambda ).
\end{align*}
In combination with \eqref{Eq:approx_in_lambda} and \eqref{Eq:approx_in_r}, and the assumption \eqref{Eq:ass_lambda_r}, we obtain
\begin{align}
\nonumber
     \int_{0}^{T_{r,\lambda}}x_{r,\lambda}\mathrm{d}t & =\lambda \int_0^{T_{r,\lambda}}(1-\cos(t))\mathrm{d}t+o_{\lambda,r\rightarrow 0}(\lambda )\\
     \label{Eq:lambda_pi}
    & =\lambda \pi+o_{\lambda,r\rightarrow 0}\left(\lambda\right).
\end{align}

Returning to the solution $x=x_{\lambda,x(0)}$ of \eqref{Eq:EL_eq_orthogonal}, we obtain by \eqref{Eq:lambda_pi} together with the bounds $|\lambda|\leq \|x\|_\infty\leq C\epsilon^{1/4}$ that
\begin{align*}
    0=\int x(t)\mathrm{d}t=\lambda \left(\pi+o_{\epsilon\rightarrow 0}(1)\right),
\end{align*}
i.e. for $\epsilon$ small enough we conclude $\lambda=0$. \\

Finally, let us discuss the Hessian $\nabla^2 \widetilde{F}$. At $x=x_0$, we have 
\begin{align*}
  \nabla^2 \widetilde{F}(x_0)=(-\Delta)^{-1}(-\Delta-1),  
\end{align*}
which, as an operator on $\dot{H}^1((0,\ell))$, has exactly a single negative eigenvalue as long as $\ell=\pi+\epsilon\in (\pi,2\pi)$. For $x\in \{x_1,x_2\}$, let us assume w.l.o.g. that $r:=x(0)>0$. By \eqref{Eq_compare_harmonic}
\begin{align*}
    x(t)=r\cos(t)+o_{r\rightarrow 0}(r),
\end{align*}
and $T_{0,r}=\pi+\epsilon$, and in particular \eqref{Eq:for_T_0_r} implies
\begin{align*}
    r=\sqrt{\frac{8\epsilon}{3\pi}}+o_{\epsilon\rightarrow 0}(\sqrt{\epsilon}).
\end{align*}
In the following it is enough to verify that $0$ is not an eigenvalue of $\nabla^2 F(x)$, i.e. $\partial_w^2 F(x)>0$ for all $w\in \dot{H}\setminus \{0\}$. Using the operator 
\begin{align*}
  Z=P(-\Delta -1+ 3x^2)P ,
\end{align*}
where $P(w):=w-\frac{1}{\pi+\epsilon}\int_0^{\pi+\epsilon} w(s)\mathrm{d}s$, we can express
\begin{align*}
    \partial_w^2 F(x)=\langle w,Zw\rangle_{L^2}.
\end{align*}
We view $Z$ as a perturbation of $P(-\Delta -1)P$, which has the lowest eigenvalue 
\begin{align*}
   \left(\frac{\pi}{\pi+\epsilon}\right)^2-1=-\frac{2\epsilon}{\pi}+O_{\epsilon\rightarrow 0}(\epsilon^2) 
\end{align*}
and corresponding eigenstate 
\begin{align*}
  w(t):=\sqrt{\frac{2}{\pi+\epsilon}}\cos\! \left(\frac{\pi}{\pi+\epsilon}t\right)  .
\end{align*}
 Applying the Kato-Temple inequality, and using that the second eigenvalue of the unperturbed operator $P(-\Delta -1)P$ satisfies $1-o_{\epsilon\rightarrow 0}(1)$, we obtain that the lowest eigenvalue of $Z$ is given by
\begin{align*}
    \inf \sigma(Z) & =\left\langle w,Zw\right\rangle_{L^2}+O_{\epsilon\rightarrow 0}(\epsilon^2)\\
    & =\int_0^{\pi+\epsilon}\left((w')^2+w^2\right)\mathrm{d}t+3\int_0^{\pi+\epsilon}w(t)^2x(t)^2\mathrm{d}t+O_{\epsilon\rightarrow 0}(\epsilon^2)\\
    &= -\frac{2\epsilon}{\pi}+3\int_0^{\pi+\epsilon}\left(\sqrt{\frac{2}{\pi}}\cos (t)\right)^2 \left(\sqrt{\frac{8\epsilon}{3\pi}}\cos(t)\right)^2\mathrm{d}t+O_{\epsilon\rightarrow 0}(\epsilon^2)\\
    &= \frac{4 \epsilon}{\pi}+O_{\epsilon\rightarrow 0}(\epsilon^2),
\end{align*}
which concludes the proof for $\epsilon$ small enough.

\end{proof}

With Proposition \ref{Prop:Critical_points_CH} at hand, we are in a position to verify Corollary \ref{Cor:CH_result}.

\begin{corollary}
\label{Cor:CH_result}
Let $0<\ell<\pi$. Then
\[
   \liminf_{\beta\to\infty}\lambda_{\beta,1}^{(\mathrm{CH})}>0 .
\]
Furthermore, there exists $\epsilon_0>0$ such that, for
$\pi<\ell<\pi+\epsilon_0$,
\[
   \liminf_{\beta\to\infty}\lambda_{\beta,2}^{(\mathrm{CH})}>0 ,
\]
and, as $\beta\to\infty$,
\[
   \lambda_{\beta,1}^{(\mathrm{CH})}
   =
   \frac1\pi
   \frac{(\ell^2-\pi^2)\pi^2}{\ell^4}
   \sqrt{
      \frac{\det[\nabla^2\widetilde{F}(x_1)]}
           {|\det[\nabla^2\widetilde{F}(x_0)]|}
   }\,
   e^{-\beta(\widetilde{F}(x_0)-\widetilde{F}(x_1))}
   \left(1+O(\beta^{-1/2})\right).
\]
\end{corollary}
\begin{proof}
We note that the additional requirements in Theorem \ref{Th:precise_Asy} are met, as $\mathrm{dom}(\widetilde{A})\subseteq \mathrm{dom}(\widetilde{K})\subseteq (\widetilde{B})^*$ and therefore $\eta_x=\widetilde{A}^{-1}\nu_x\in (\widetilde{B})^*$ and 
\begin{align*}
    \widetilde{A}^{-1}\! \left((\widetilde{B})^*\right)\cap (\widetilde{B})^*\supseteq \mathrm{dom}\!\left(\widetilde{A}^2\right),
\end{align*}
which is dense in $\dot{H}^1((0, \ell))$. We further notice that \eqref{eq:assump_comp} is satisfied, by the bound in \eqref{Eq:comp_bound_A_square}. Considering first the case $\pi<\ell<\pi+\epsilon_0$, we know by Proposition \ref{Prop:Critical_points_CH}, for $\epsilon_0$ small enough, that $x_1$ and $x_2$ are the minima of $\widetilde{F}$ and that $\nabla^2\widetilde F(x)$ is invertible for all critical points $x\in \{x_0,x_1,x_2\}$. In case $0<\ell<\pi$ the situation is drastically more simple, as $x_0$ is then the only minimum and only critical point of $\widetilde{F}$, with
$\nabla^2\widetilde F(x_0)\geq c>0$. In both cases, Assumption \ref{Ass_for_theorem_II} is satisfied and the residual conditions in Assumption \ref{Ass_for_theorem_I} can be verified as in the proof of Corollary \ref{Cor:AC_result}.
\end{proof}

\section*{Acknowledgements}
\noindent
MB gratefully acknowledges funding from the UZH Postdoc Grant FK-25-103. GDG is grateful to Martin Grothaus for inspiring discussions and to the Bernoulli Center in Lausanne for its warm hospitality during several visits in the course of this work.

\end{document}